\documentclass[a4paper]{amsart}
\usepackage{amscd}
\usepackage{amsthm}
\usepackage{graphicx}
\usepackage{amsmath}
\usepackage{amsfonts}
\usepackage{amssymb}
\providecommand{\U}[1]{\protect\rule{.1in}{.1in}}
\newtheorem{theorem}{Theorem}

\newtheorem{corollary}[theorem]{Corollary}

\newtheorem{lemma}[theorem]{Lemma}

\newtheorem{proposition}[theorem]{Proposition}
\newtheorem{remark}[theorem]{Remark}

\begin{document}

\title[A decay estimate for the Fourier transform of certain singular measures]{A decay estimate for the Fourier transform of certain singular measures in $\mathbb{R}^{4}$ and applications}
\author{Tom\'as Godoy and Pablo Rocha}
\address{FaMAF, Universidad Nacional de C\'ordoba, Ciudad Universitaria, 5000 C\'ordoba, Argentina.}
\email{godoy@famaf.unc.edu.ar, rp@famaf.unc.edu.ar}
\thanks{\textbf{2.020 Math. Subject Classification}: 42B20, 42B10.}

\begin{abstract}
We consider, for a class of functions $\varphi : \mathbb{R}^{2} \setminus \{ {\bf 0} \} \to \mathbb{R}^{2}$ satisfying a nonisotropic homogeneity condition, the Fourier transform $\widehat{\mu}$ of the Borel measure on $\mathbb{R}^{4}$ defined by 
\[
\mu(E) = \int_{U} \chi_{E}(x, \varphi(x)) \, dx
\]
where $E$ is a Borel set of $\mathbb{R}^{4}$ and $U = \{ (t^{\alpha_1}, t^{\alpha_2}s) : c < s < d, \, 0 < t < 1 \}$. The aim of this article is to give a decay estimate for $\widehat{\mu}$, for the case where the set of nonelliptic points of $\varphi$ is a curve in 
$\overline{U} \setminus \{ {\bf 0} \}$. From this estimate we obtain a restriction theorem for the usual Fourier transform to the graph of 
$\left. \varphi \right|_{U} : U \to \mathbb{R}^{2}$. We also give $L^{p}$-improving properties for the convolution operator 
$T_{\mu} f = \mu \ast f$.
\end{abstract}

\maketitle

\section{Introduction}

Let $U$ be an open bounded set in $\mathbb{R}^{n}$ and $\varphi : U \to \mathbb{R}^{m}$ be a continuous function. Let $\mu_U$ be the Borel measure on $\mathbb{R}^{n+m}$ supported on the graph of $\varphi$, defined by

\begin{equation} \label{mu}
\mu_U (E) = \int_{U} \chi_E (x, \varphi(x)) \, dx,
\end{equation}
\\
where $dx$ denotes the Lebesgue measure on $\mathbb{R}^{n}$. Let $\Sigma$ be the graph of $\varphi$, i.e.

\begin{equation} \label{variedad sigma}
\Sigma = \left\{ (x, \varphi(x)) : x \in U \right\}.
\end{equation}
\\
For $f \in \mathcal{S}(\mathbb{R}^{n+m})$, let $T_{\mu_U} f = \mu_U \ast f$ and let $\mathcal{R}f = \widehat{f} \, |_{\Sigma}$, 
where $\widehat{f}$ denotes the usual Fourier transform of $f$. The $L^{p}(\mathbb{R}^{n+m}) - L^{q}(\Sigma)$ boundedness of the restriction 
operator $\mathcal{R}$ to the sub-manifold $\Sigma \subset \mathbb{R}^{n+m}$ has been widely studied, in different cases, by many authors.
For $S^{1} \subset \mathbb{R}^{2}$ by C. Fefferman in \cite{Feff}, for $S^{n-1} \subset \mathbb{R}^{n}$, $n \geq 2$, by P. Tomas and 
E. Stein in \cite{Tomas} and \cite{Stein2} respectively, for quadratic surfaces with non vanishing gaussian curvature by R. Strichartz in \cite{Str}. The case of compact $n-$dimensional manifolds $\Sigma \subset \mathbb{R}^{2n}$ was studied by E. Prestini in \cite{Pres} under the assumption that: for each $x_0 \in \Sigma$ there exists a local chart $X(x_1, ..., x_n)$ around $x_0$ satisfying that the vectors 
$\left\{ \displaystyle{\frac{\partial X}{\partial x_i}} \right\}_{i=1}^{n}$ and 
$\left\{ \displaystyle{\frac{\partial^{2} X}{\partial x_i \partial x_j}} \right\}_{i, j=1}^{n}$ span $\mathbb{R}^{2n}$.
The case of two dimensional manifolds $\Sigma = \left\{ (x, \varphi(x)) : x \in U \subset \mathbb{R}^{2}\right\} \subset \mathbb{R}^{4}$
was studied by M. Christ in \cite{Christ}, where a restriction theorem is given under the assumption that: for each $x \in U$ and 
$\theta \in \mathbb{R}$, if $\det {\bf H}_{x} \left( \left\langle \varphi(x), (\cos(\theta), \sin(\theta)) \right\rangle \right) = 0$
(where $\langle \, , \,  \rangle$ denotes the inner product in $\mathbb{R}^{2}$ and ${\bf H}_{x}$ denotes the Hessian matrix of second partial derivatives with respect to $x$) then
$\frac{d}{d\theta}\det {\bf H}_{x} \left(\left\langle \varphi(x), (\cos(\theta), \sin(\theta)) \right\rangle \right) \neq 0$. More restriction theorems for homogeneous sub-manifolds of $\mathbb{R}^{n}$ can be found in \cite{Carli2} and \cite{Della}. A very interesting book about the Fourier restriction problem can be found in \cite{Demeter}.

On the other hand, the $L^{p}(\mathbb{R}^{n+m}) - L^{q}(\mathbb{R}^{n+m})$ boundedness of the convolution operator $T_{\mu_U}$ has
received also considerable attention in the literature. Define the type set $E_{\mu_U}$ as the set of the pairs
$\left( \frac{1}{p}, \frac{1}{q} \right) \in [0,1] \times [0,1]$ such that $\|T_{\mu_U} f\|_{q} \leq C_{p,q} \| f \|_{p}$ for all
$f \in \mathcal{S}(\mathbb{R}^{n+m})$, where $C_{p,q}$ is a positive constant depending only on $p$ and $q$. Explicit descriptions of the set
$E_{\mu_U}$ are known for many cases (see e.g. \cite{Fulvio} and the references therein).

In particular, S. W. Drury and K. Guo in \cite{Drury} studied the case $n=m$, for $\varphi = (\varphi_1, ..., \varphi_n)$ regular enough in an open $V \subset \mathbb{R}^{n}$, and they proved that if each point in $V$ is elliptic for $\varphi$ and $U$ is a bounded open set such that $\overline{U} \subset V$, then the set $E_{\mu_U}$ is the closed triangle with vertices $(0,0)$, $(1,1)$ and $\left(\frac{2}{3}, \frac{1}{3} \right)$. Elliptic points can be defined as follows: for $x \in V$, consider the function $Q_x$ on $\mathbb{R}^{n}$ defined by
\begin{equation} \label{Qx}
Q_{x}(\zeta) = \det\left( \sum_{j=1}^{n} \zeta_j \varphi''_j (x)  \right), \,\,\,\,\,\,\,\, \zeta = (\zeta_1, ..., \zeta_n) \in 
\mathbb{R}^{n},
\end{equation}
where $\varphi''_j (x) := ({\bf H} \varphi_j)(x)$ is the Hessian matrix of the function $\varphi_j$ at $x$. We point out that $Q_x$ is only a quadratic form for $n=2$. Then, we say that $x \in V$ is elliptic if $\displaystyle{\min_{\zeta \in S^{n-1}}} |Q_x(\zeta)| > 0$, i.e.: if for all $\zeta \in S^{n-1}$ the surface 
$\Sigma_{\zeta} = \{ (y, \langle \varphi(y), \zeta \rangle) : y \in V\}$ has nonzero curvature at $x$ (although the definition of elliptic point used in \cite{Drury} is different from the given above, at least for $n=2$, they are equivalent; see the comments after Lemma \ref{elliptic point} below). Similar results about the type set $E_{\mu_U}$ are given in \cite{Godoy} for the case when $\varphi$ is a non isotropic homogeneous function such that the origin is the unique nonelliptic point.

Our aim in this paper is to study the case where $n=m=2$ and $\varphi$ is a function with non isotropic homogeneity, whose set of nonelliptic points is a curve in $\mathbb{R}^{2}$.

Now, we state our precise assumptions and results. Let $\alpha_1, \, \alpha_2 >0$ with $\alpha_1 \neq \alpha_2$, and for $t>0$, 
$x=(x_1, x_2) \in \mathbb{R}^{2}$, let
\begin{equation} \label{dilation}
t \bullet x = \left( t^{\alpha_1} x_1, t^{\alpha_2} x_2 \right).
\end{equation} 
For $a < b$, we set

\[
V^{a,b} = \left\{ t \bullet (1,s) : s \in (a,b) \,\, \text{and} \,\,  t > 0 \right\},
\]
\\
and for $a < c < d < b$, we put

\begin{equation} \label{set Vcd1}
V^{c,d}_{1} = \left\{ t \bullet (1,s) : s \in (c,d) \,\, \text{and} \,\,  0 < t < 1 \right\},
\end{equation} \\
and let $\varphi =(\varphi_1, \varphi_2) : V^{a,b} \to \mathbb{R}^{2}$ be a function. We assume that the following hypotheses hold:

\

H1) $\varphi$ is a real analytic function on $V^{a,b}$.

H2) For some $m \geq 3 (\alpha_1 + \alpha_2)$, $\varphi(t \bullet x)= t^{m} \varphi(x)$ for all $x \in V^{a,b}$, and all $t > 0$.

H3) For $a < c < d < b$ and some $\sigma \in [c, d]$, the set of nonelliptic  points for $\varphi$ in 
$\overline{V^{c,d}} \setminus \{ {\bf 0}\}$ is the curve  $\{ t \bullet (1, \sigma) : t>0 \}$.

\

Under these assumptions, from Lemma \ref{elliptic point} and Lemma \ref{eigen functions} (see Preliminaries below), there exist two positive integers $n_1$ and $n_2$, and a positive constant $D$ such that for $\delta$ positive and small enough

\begin{equation} \label{estim 1}
\displaystyle{\min_{\zeta \in S^{1}}} |Q_{t \bullet (1,s)}(\zeta)| \geq D \, t^{\beta} |s - \sigma|^{n_1} \,\, \text{for all} \,\, s \in 
(\sigma - \delta, \sigma), \,\, \text{and all} \,\, t > 0, \,\,\,\, \text{if} \,\, \sigma \in (c,d],
\end{equation}

and
\begin{equation} \label{estim 2}
\displaystyle{\min_{\zeta \in S^{1}}} |Q_{t \bullet (1,s)}(\zeta)| \geq D \, t^{\beta} |s - \sigma|^{n_2} \,\, \text{for all} \,\, s \in 
(\sigma, \sigma + \delta), \,\, \text{and all} \,\, t > 0, \,\,\,\, \text{if} \,\, \sigma \in [c,d).
\end{equation} \\
where $\beta = 2(m - \alpha_1 - \alpha_2)$. Taking into account the definition of elliptic points, (\ref{estim 1}) and (\ref{estim 2}) suggest
to use the number $\max \{ n_1, n_2 \}$ as a measure of the degeneracy of ellipticity along the curve $\{ t \bullet (1, \sigma) : t>0 \}$.

\

Let $a < c < d < b$ such that $\sigma \in [c,d]$ ($\sigma$ satisfies H3), and let $\mu$ be the measure, fixed from now on, defined by (\ref{mu}) taking there 
$U = V^{c,d}_{1}$, where $V^{c,d}_{1}$ is given by (\ref{set Vcd1}). Let $\widehat{\mu}$ be its Fourier transform given, for 
$\xi', \xi'' \in \mathbb{R}^{2}$, by
\[
\widehat{\mu}(\xi', \xi'') = \int_{V^{c,d}_{1}} e^{-i\left( \langle x, \xi' \rangle + \langle \varphi(x), \xi'' \rangle  \right)} \, dx.
\]

In Section 3, Theorem \ref{muhat estimate}, we prove that $\widehat{\mu}$ satisfies, for some constant $C>0$, the following estimate

\begin{equation} \label{estim 3}
\left|  \widehat{\mu}(\xi', \xi'') \right| \leq C \, |\xi''|^{-\frac{\alpha_1+\alpha_2}{m}}, \,\,\,\,\,\, \text{for all} \,\,
\xi' \in \mathbb{R}^{2} \,\, \text{and all} \,\, \xi'' \in \mathbb{R}^{2} \setminus \{ {\bf 0} \}.
\end{equation}  

\

In Section 4, we consider the restriction operator $\mathcal{R}f = \widehat{f} \, |_{\Sigma}$, with $\Sigma$ given by 
(\ref{variedad sigma}) (with $U = V^{c,d}_{1}$). Following ideas in \cite{Christ}, using (\ref{estim 3}) and complex interpolation for a suitable analytic family of operators, in Theorem \ref{Restrict thm}, we prove that there exists a constant $C >0$ such that 
$\| \mathcal{R}f \|_{L^{2}(\Sigma)} \leq C \, \| f \|_{L^{p}(\mathbb{R}^{4})}$ for all $f \in \mathcal{S}(\mathbb{R}^{4})$ if and only if
$\displaystyle{\frac{\alpha_1 + \alpha_2 + 4m}{2(\alpha_1 + \alpha_2 + 2m)}} \leq \frac{1}{p} \leq 1$. Finally, using again (\ref{estim 3}),
we study the type set $E_{\mu}$ for the convolution operator $T_{\mu}f = \mu \ast f$, for $f \in \mathcal{S}(\mathbb{R}^{4})$. Theorem 
\ref{Type set thm} states that, for $p = \displaystyle{\frac{\alpha_1 + \alpha_2 + 2m}{\alpha_1 + \alpha_2 + m}}$, the closed triangle with vertices $(0,0)$, $(1,1)$ and $\left(\frac{1}{p}, \frac{1}{p'} \right)$ is contained in $E_{\mu}$; moreover 
$\left(\frac{1}{p}, \frac{1}{p'} \right) \in \partial E_{\mu}$. Finally, in Section 5, we give an example of a function 
$\varphi =(\varphi_1, \varphi_2) : V^{a,b} \to \mathbb{R}^{2}$ satisfying the hypotheses required.

\section{Preliminaries}

For the sequel, we consider a function $\varphi = (\varphi_1, \varphi_2) : V^{a,b} \subset \mathbb{R}^{2} \setminus \{ {\bf 0} \} \to 
\mathbb{R}^{2}$ satisfying H1-H3. Given $x \in V^{a,b}$, let $Q_x$ be the quadratic form on $\mathbb{R}^{2}$ defined by

\begin{equation} \label{Quad}
Q_x(\zeta) = det \left( \zeta_1 \varphi_{1}''(x) + \zeta_2 \varphi_{2}''(x) \right), \,\,\,\,\, \zeta = (\zeta_1, \zeta_2) \in \mathbb{R}^{2},
\end{equation} \\
where, for each $j=1,2$, $\varphi''_j (x) := ({\bf H} \varphi_j)(x)$ is the Hessian matrix of the function 
$\varphi_j : V^{a,b} \to \mathbb{R}$ at $x$; and let $K(x) = \left( k_{ij}(x) \right)$ be its associated symmetric matrix, i.e.: satisfying

\begin{equation} \label{matrix K}
Q_x(\zeta) = \left\langle K(x) \zeta, \zeta \right\rangle, \,\,\,\,\, \zeta = (\zeta_1, \zeta_2) \in \mathbb{R}^{2}.
\end{equation}
\\
We say that a point $x \in V^{a,b}$ is elliptic for $\varphi$, if $\displaystyle{\min_{\zeta \in S^{1}}} |Q_x(\zeta)| > 0$.

\begin{lemma} \label{elliptic point}
If $x \in V^{a,b}$ is a elliptic point for $\varphi$, then the eigenvalues of $K(x)$ are either both positive or both negative and
\[
\displaystyle{\min_{\zeta \in S^{1}}} |Q_x(\zeta)|= \min \{ |\Lambda| : \Lambda \, \text{is an eigenvalue of} \,\, K(x) \}.
\]
\end{lemma}

\begin{proof} Follows immediately from the definition of elliptic point and (\ref{matrix K}).
\end{proof}

Let us quote that a definition of elliptic point, different from the stated at the introduction, is given in \cite{Drury}. Let us recall it:
For $x \in V^{a,b}$, $h \in \mathbb{R}^{2}$, let $\varphi''(x) h$ be the $2 \times 2$ matrix whose $j-$th column is $\varphi_{j}''(x) h$, where $\varphi''_j (x)$ is the Hessian matrix of the function $\varphi_j : V^{a,b} \to \mathbb{R}$ at $x$. A point $x \in V^{a,b}$ is called elliptic in \cite{Drury}, if $\det \left( \varphi''(x) h \right) \neq 0$ for all $h \in \mathbb{R}^{2} \setminus \{ {\bf 0} \}$. This definition is equivalent to ours. Indeed, consider the symmetric matrix $B(x)$ defined by 
$\langle B(x) \zeta, \zeta \rangle = \det \left( \varphi''(x) \zeta \right)$, $\zeta \in \mathbb{R}^{2}$. An explicit computation shows that $\det B(x) = \det K(x)$ and so the two definitions agree.

\begin{lemma}\label{eigen functions}
$(i)$ There exist two homogeneous functions $\Lambda_1$ and $\Lambda_2$ on $V^{a,b}$ of degree $\beta = 2(m - \alpha_1 - \alpha_2)$ respect to the dilations given in $(\ref{dilation})$, which give the eigenvalues of $K(x)$. Moreover, $\Lambda_1$ and $\Lambda_2$ result real analytic functions on $V^{a,b}$. \\
$(ii)$ For $\delta > 0$ and small enough, it holds that: \\
for some $j=1,2$, $\min\{|\Lambda_i(x)| : i=1,2 \} = |\Lambda_j(x)|$ for all $x \in V^{\sigma-\delta, \sigma}$, if $\sigma \in (c,d]$, and \\
for some $k=1,2$, $\min\{|\Lambda_i(x)| : i=1,2 \} = |\Lambda_k(x)|$ for all $x \in V^{\sigma, \sigma + \delta}$, if $\sigma \in [c,d)$. \\
$(iii)$ There exist two positive integers $n_1$ and $n_2$, and a positive constant $D$ such that for $\delta$ positive and small enough
\[
\min \{ |\Lambda_i(t \bullet (1,s))| : i=1,2\} \geq D \, t^{\beta} |s - \sigma|^{n_1} \,\, \text{for all} \,\, s \in 
(\sigma - \delta, \sigma), \,\, \text{and all} \,\, t > 0, \,\,\,\, \text{if} \,\, \sigma \in (c,d]
\]
and
\[
\min \{ |\Lambda_i(t \bullet (1,s))| : i=1,2\} \geq D \, t^{\beta} |s - \sigma|^{n_2} \,\, \text{for all} \,\, s \in 
(\sigma, \sigma + \delta), \,\, \text{and all} \,\, t > 0, \,\,\,\, \text{if} \,\, \sigma \in [c,d),
\]
with $\beta = 2(m - \alpha_1 - \alpha_2)$. \\
$(iv)$ If $I \subset [c,d]$ is a closed interval such that $\sigma \notin I$, then there exists a positive constant $\tilde{D}$ such that
$\min \{ |\Lambda_i(t \bullet (1,s))| : i=1,2\} \geq \tilde{D} \, t^{\beta}$ for all $s \in I$ and all $t > 0$.
\end{lemma}

\begin{proof}
A computation shows that the entries $k_{ij}(x)$ of $K(x)$ are given by

\[
k_{11}(x) = \det \varphi_{1}''(x), \,\,\,\,\,\,\,  k_{22}(x) = \det \varphi_{2}''(x),
\]

\[
k_{12}(x) = k_{21}(x) = \frac{1}{2} \left( \frac{\partial^{2} \varphi_1}{\partial x_{1}^{2}} \, 
\frac{\partial^{2} \varphi_2}{\partial x_{2}^{2}} +
\frac{\partial^{2} \varphi_1}{\partial x_{2}^{2}} \, \frac{\partial^{2} \varphi_2}{\partial x_{1}^{2}} - 
2 \frac{\partial^{2} \varphi_1}{\partial x_{1} \partial x_{2}} \, \frac{\partial^{2} \varphi_2}{\partial x_{1} \partial x_{2}} \right)(x).
\]
\\
From H1 it follows that the functions $k_{ij}$, $1 \leq i,j \leq 2$, are real analytic on $V^{a,b}$, and from H2 the $k_{ij}$'s result homogeneous of degree $\beta = 2(m - \alpha_1 - \alpha_2)$, i.e.: for every $1 \leq i,j \leq 2$, $k_{ij}(t \bullet x) = t^{\beta} k_{ij}(x)$ 
for $x \in V^{a.b}$, and $t > 0$. It is easy to check that the eigenvalues $\Lambda_1(x)$ and $\Lambda_2(x)$ of $K(x)$ are given by
\[
\Lambda_1(x) = \frac{1}{2} \left( k_{11}(x) + k_{22}(x) - \sqrt{4 k_{12}^{2}(x) + (k_{11}(x) - k_{22}(x))^{2}} \right),
\]
and
\[
\Lambda_2(x) = \frac{1}{2} \left( k_{11}(x) + k_{22}(x) + \sqrt{4 k_{12}^{2}(x) + (k_{11}(x) - k_{22}(x))^{2}} \right).
\]

\

The homogeneity of $\Lambda_1$ and $\Lambda_2$ follow from that of the entries $k_{ij}$. We observe that for every $s_0 \in (a,b)$ fixed and $\delta$ positive and small enough, each $\Lambda_i(1, \, \cdot \,)$ is real analytic on $(s_0 - \delta, s_0 + \delta)$. Indeed, this is clear if either $k_{12}(1, s_0) \neq 0$ or $k_{11}(1, s_0) \neq k_{22}(1, s_0)$. Now, if $k_{12}(1, s_0) = 0$ and 
$k_{11}(1, s_0) - k_{22}(1, s_0) = 0$, we consider the function

\begin{equation} \label{kaes}
f(s) = 4 k_{12}^{2}(1,s) + (k_{11}(1,s) - k_{22}(1,s))^{2}.
\end{equation}
\\
If $f$ is identically zero on $(s_0 - \delta, s_0 + \delta)$ for some $\delta > 0$, then $\Lambda_1(1,s) = \Lambda_2(1,s) = k_{11}(1,s)$. So that each $\Lambda_i(1, \, \cdot \,)$ is real analytic on $(s_0 - \delta, s_0 + \delta)$. If $f$ is not identically zero, since $f \geq 0$ and
$f(s_0) = 0$, from the analyticity of the entries $k_{ij}$ it follows that there exist $\delta >0$, a positive integer $q$, and a nonnegative real analytic function $h$ such that $h(s_0) > 0$ and $f(s) = (s - s_0)^{2q} [h(s)]^{2}$ for all $s \in (s_0 - \delta, s_0 +\delta)$. Now, for $i=1,2$, we set

\[
\widetilde{\Lambda}_i(1,s) = \frac{1}{2} \left( k_{11}(1,s) + k_{22}(1,s) + (-1)^{i} (s - s_0)^{q} h(s) \right).
\]
\\
It is clear that the functions $\widetilde{\Lambda}_i(1, \, \cdot \,)$ are real analytic on $(s_0 - \delta, s_0 +\delta)$. If $q$ is even, then $\widetilde{\Lambda}_i(1,s) = \Lambda_i(1,s)$ for $s \in (s_0 - \delta, s_0 +\delta)$ and $i=1,2$. If $q$ is odd, then 
$\widetilde{\Lambda}_1(1,s) = \Lambda_1(1,s)$ for $s \in [s_0, s_0 +\delta)$, and $\widetilde{\Lambda}_1(1,s) = \Lambda_2(1,s)$ for 
$s \in (s_0 - \delta, s_0)$ (similar identities are obtained for $\widetilde{\Lambda}_2(1, \, \cdot \,)$). Thus, the functions 
$x \to \widetilde{\Lambda}_i(x)$, $i=1,2$, give the eigenvalues of $K(x)$. From this analysis and the homogeneity of $\Lambda_1$ and 
$\Lambda_2$ we obtain its analyticity on $V^{a,b}$. Hence, $(i)$ follows.

To see $(ii)$, we observe that each point in $V^{c, \sigma}$ (resp. in $V^{\sigma, d}$) is elliptic and so, by Lemma \ref{elliptic point}, 
$\Lambda_1$ and $\Lambda_2$ have the same sign in $V^{c, \sigma}$ (resp. in $V^{\sigma, d}$). Then, $(ii)$ follows from the fact that 
$\Lambda_1 \leq \Lambda_2$ on $V^{c, \sigma}$ (resp. on $V^{\sigma, d}$).

Finally, Lemma \ref{elliptic point}, $(i)$ and $(ii)$ allow us to get $(iii)$ and $(iv)$. 
\end{proof}

\begin{lemma} \label{Gs}
Let $G : S^{1} \times (a,b) \to \mathbb{R}$ be the function given by $G(\zeta,s) = \langle \varphi(1,s), \zeta \rangle$. Then,

\begin{equation} \label{G function}
\alpha_{1}^{2} \det \left( \zeta_1 \varphi_{1}''(1,s) + \zeta_2 \varphi_{2}''(1,s) \right) =
\end{equation}
\[
m(m-\alpha_1)(G G_{ss})(\zeta,s) + \alpha_2 (\alpha_1 - \alpha_2) s (G_s G_{ss})(\zeta,s) - (m - \alpha_2)^{2} G_{s}^{2}(\zeta,s),
\]
\\
for every $\zeta =(\zeta_1, \zeta_2) \in S^{1}$ and every $s \in (a,b)$.
\end{lemma}

\begin{proof} For $\zeta \in S^{1}$ fixed, let $\Psi(x)= \langle \varphi(x), \zeta \rangle$, for $x \in V^{a,b}$. Clearly, the function $\Psi$ is homogeneous of degree $m$ respect to the dilations in (\ref{dilation}) and so, from the Euler equations, 
\[
m\Psi(x) = \alpha_1 x_1 \Psi_{x_1}(x) + \alpha_2 x_2 \Psi_{x_2}(x)
\]
\[
(m-\alpha_1)\Psi_{x_1}(x) = \alpha_1 x_1 \Psi_{x_1 x_1}(x) + \alpha_2 x_2 \Psi_{x_1 x_2}(x)
\]
\[
(m-\alpha_2)\Psi_{x_2}(x) = \alpha_1 x_1 \Psi_{x_1 x_2}(x) + \alpha_2 x_2 \Psi_{x_2 x_2}(x).
\]
Thus
\begin{equation} \label{G1}
mG(\zeta, s) = \alpha_1 \Psi_{x_1}(1,s) + s \alpha_2 \Psi_{x_2}(1,s) =
\end{equation}
\[
\frac{\alpha_{1}^{2}}{m-\alpha_1} \Psi_{x_1 x_1}(1,s) + \frac{\alpha_{2}^{2} s^{2}}{m-\alpha_2} \Psi_{x_2 x_2}(1,s) +
\left(\frac{\alpha_{1} \alpha_{2}}{m-\alpha_1} + \frac{\alpha_{1} \alpha_{2}}{m-\alpha_2} \right) s \, \Psi_{x_1 x_2}(1,s);  
\]
\\
we also have that
\begin{equation} \label{G2}
\Psi_{x_2}(1,s) = G_{s}(\zeta, s) \,\,\,\, \text{and} \,\,\,\, \Psi_{x_2 x_2}(1,s) = G_{ss}(\zeta, s),
\end{equation}
and so
\[
(m-\alpha_2) \, G_{s}(\zeta, s) = \alpha_1 \Psi_{x_1 x_2}(1,s) + \alpha_2 \, s \, G_{ss}(\zeta, s).
\]
Then
\begin{equation} \label{G3}
\Psi_{x_1 x_2}(1,s) = \frac{m-\alpha_2}{\alpha_1} G_{s} (\zeta,s) - \frac{\alpha_2}{\alpha_1} s \, G_{ss}(\zeta,s),
\end{equation}
from (\ref{G1}), using (\ref{G2}) and (\ref{G3}) we can express $\Psi_{x_1 x_1}(1,s)$ in terms of $G(\zeta,s)$, $G_{ss}(\zeta,s)$ and 
$G_{ss}(\zeta,s)$. Taking into account this expression, and also (\ref{G2}) and (\ref{G3}), a computation of $\det(\Psi''(1,s))$ gives
(\ref{G function}).
\end{proof}

\begin{remark} \label{lower bound}
Let $n_1$, $n_2$, $D$ and $\tilde{D}$ be as in Lemma \ref{eigen functions}. For $(\zeta, s) \in S^{1} \times (a,b)$, let

\[
H(\zeta, s) = m(m-\alpha_1) (GG_{ss})(\zeta,s) + \alpha_2 (\alpha_1 - \alpha_2) s (G_s G_{ss})(\zeta,s) - (m - \alpha_2)^{2} 
G_{s}^{2}(\zeta,s).
\]
\\
From (\ref{Quad}), and Lemmas \ref{elliptic point}, \ref{eigen functions} and \ref{Gs} we have that 

\

\, $(i)$ For $\delta > 0$ and small enough, $|H(\zeta,s)| \geq D|s- \sigma|^{n_1}$ for all $s \in (\sigma - \delta, \sigma)$ and all $\zeta \in S^{1}$, if $\sigma \in (c,d]$.

\, $(ii)$ For $\delta > 0$ and small enough, $|H(\zeta,s)| \geq D|s- \sigma|^{n_2}$ for all $s \in (\sigma, \sigma + \delta)$ and all 
$\zeta \in S^{1}$, if $\sigma \in [c,d)$.

\, $(iii)$ If $I \subset [c,d] \setminus \{ \sigma \}$ is a closed interval, then $|H(\zeta,s)| \geq \tilde{D}$ for all $s \in I$ and all 
$\zeta \in S^{1}$.
\end{remark}

\section{An estimate for $\widehat{\mu}$}

For an open set $U \subset \mathbb{R}^{2}$, we put $\mu = \mu_U$ for the measure defined by (\ref{mu}). For $[c,d] \subset (a,b)$ such that 
$\sigma \in [c,d]$ (see hypothesis H3), we take the open set $U=V_{1}^{c, d}$ given by (\ref{set Vcd1}). Then, the change of variable 
$x_1 = t^{\alpha_1}$, $x_2 = s t^{\alpha_2}$ gives

\begin{equation} \label{muhat}
\widehat{\mu} (\xi', \xi'') = \alpha_1 \int_{c}^{d} \left( \int_{0}^{1} e^{-i(\xi_1 t^{\alpha_1} + \xi_2 s 
t^{\alpha_2} + t^{m} \langle \varphi(1,s), \xi'' \rangle)} \, t^{\alpha_1 + \alpha_2 -1} \,  dt  \right) \, ds,
\end{equation}
\\
where $\xi'=(\xi_1, \xi_2)$ and $\xi''=(\xi_3, \xi_4)$.

\

The following lemmas will allow us to give a decay estimate for (\ref{muhat}).

\begin{lemma} \label{int osc} 
Let $A_j \in \mathbb{R}$, for $j=1,2,3$. Then there exists a positive constant $C$ such that 
\begin{equation} \label{int osc1}
\left|  \int_{0}^{1} e^{-i(A_1 t^{\alpha_1} + A_2 t^{\alpha_2} + A_3 t^{m})} \, t^{\alpha_1 + \alpha_2 -1} \, dt \right| \leq 
C |A_3|^{-\frac{\alpha_1 + \alpha_2}{m}}.
\end{equation}
\end{lemma}

\begin{proof} Without loss of generality we can assume $A_3 > 0$. The change of variable  
$\tau = A_{3}^{\frac{\alpha_1+\alpha_2}{m}}t^{\alpha_1+\alpha_2}$ allows us to express the integral in (\ref{int osc1}) as
\[
\frac{1}{(\alpha_1+\alpha_2)A_3^{\frac{\alpha_1+\alpha_2}{m}}} \int_{0}^{A_{3}^{\frac{\alpha_1+\alpha_2}{m}}} e^{-i(a \tau^{\gamma_1} + b \tau^{\gamma_2} + \tau^{\gamma_3})} \, d\tau,
\]
with $a:=A_1 A_3^{-\frac{\alpha_1}{m}}$, $b:= A_2 A_3^{-\frac{\alpha_2}{m}}$, $\gamma_j := \frac{\alpha_j}{\alpha_1+\alpha_2}$ for $j=1,2$ 
and $\gamma_3:= \frac{m}{\alpha_1+\alpha_2}$. Let $\Phi(\tau) = a \tau^{\gamma_1} + b \tau^{\gamma_2} + \tau^{\gamma_3}$. It is clear that
$0 < \gamma_1, \gamma_2 < 1$ and $\gamma_3 \geq 3$ (see hyp. H2), so $\Phi$ do not reduce to zero after from taking three derivatives. To prove the lemma it is enough to see that there exists a positive constant $C$ independent of $a$, $b$, and $A_3$ such that for all $B > 1$
\begin{equation} \label{int osc2}
\left| \int_{1}^{B} e^{-i \Phi(\tau)} \, d\tau \right| \leq C.
\end{equation}
For $\gamma \in \mathbb{R}$ and $j \in \mathbb{N}$, we set $\gamma(j) = \gamma (\gamma - 1) \cdot \cdot \cdot (\gamma - j +1)$. We observe that $\tau^{j} \Phi^{(j)}(\tau) = \gamma_{1}(j) a \tau^{\gamma_1} + \gamma_{2}(j) b \tau^{\gamma_2} + \gamma_{3}(j) \tau^{\gamma_3}$, 
$j=1,2,3$, and that the matrix equation
\begin{gather*}
 \begin{bmatrix} \gamma_1(1) & \gamma_1(2) & \gamma_1(3) \\ \gamma_2(1) & \gamma_2(2) & \gamma_2(3) \\ \gamma_3(1) & \gamma_3(2) & 
\gamma_3(3)  \end{bmatrix}
\begin{bmatrix} c_1 \\ c_2 \\ c_3 \end{bmatrix}
 = \begin{bmatrix} 0 \\ 0 \\ 1 \end{bmatrix}
\end{gather*}
has unique solution $(c_1, c_2, c_3)$ (since $\det \left[ \gamma_i(j) \right] = \gamma_1 \gamma_2 \gamma_3(\gamma_2 - \gamma_1)
(\gamma_3 - \gamma_1) (\gamma_3 - \gamma_2) \neq 0$). Thus $\tau^{\gamma_3} = \sum_{j=1}^{3} c_j \tau^{j} \Phi^{(j)}(\tau)$ and so
$\tau^{\gamma_3} \leq \sum_{j=1}^{3} |c_j \tau^{j} \Phi^{(j)}(\tau)|$ for all $\tau \geq 1$. Then, $(1, B) \subset \cup_{j=1}^{3} I_j$ 
where $I_j = \{ \tau \in (1,B) : |\Phi^{(j)}(\tau)| \geq 1/3|c_j| \}$. Since $\frac{d}{d\tau} \left( \tau^{\gamma_1-\gamma_2+1} 
\frac{d}{d\tau} \left( \tau^{-\gamma_1 + j +1} \frac{d}{d\tau} \Phi^{(j)}(\tau) \right) \right) \neq 0$ for all $\tau \geq 1$ and each 
$j=1,2,3$, the Rolle's Theorem applied twice implies that, for each $c \in \mathbb{R}$, the equation $\Phi^{(j)}(\tau) =c$ has at most 
three solutions. Hence, each $I_j$ has at most six connected components. Being $m \geq 3(\alpha_1+\alpha_2)$ we have $\gamma_3 \geq 3$, 
so (\ref{int osc2}) follows from the Van der Corput lemma (see e.g. \cite{Stein3}, p. 332).
\end{proof}

Next, we define three functions $f_i : S^{1} \times (a,b) \to \mathbb{R}$, $i=1,2,3$, by

\[
f_1(\zeta, s) = m(m-\alpha_1) \, G(\zeta, s) \, G_{ss}(\zeta, s),
\]
\begin{equation} \label{f123}
f_2(\zeta, s) = (m-\alpha_2)^{2} \, G_{s}^{2}(\zeta, s),
\end{equation}
\[
f_3(\zeta, s) =  \alpha_2(\alpha_1-\alpha_2) \, s \, G_{s}(\zeta, s) \, G_{ss}(\zeta, s),
\]
\\
and for $\delta > 0$ and small, $\zeta \in S^{1}$ and $C > 0$ let

\begin{equation} \label{Ijotas}
\left\{\begin{array}{c}
I_{1}^{\delta, \zeta, C} = \{s \in (\sigma - \delta, \sigma) : f_1 (\zeta, s) > \frac{C}{4}|s-\sigma|^{n_1} \},\\\\
I_{2}^{\delta, \zeta, C} = \{s \in (\sigma - \delta, \sigma) : f_1 (\zeta, s) < -\frac{C}{4}|s-\sigma|^{n_1} \},\\\\
I_{3}^{\delta, \zeta, C} = \{s \in (\sigma - \delta, \sigma) : f_2 (\zeta, s) > \frac{C}{4}|s-\sigma|^{n_1} \},\\\\
I_{4}^{\delta, \zeta, C} = \{s \in (\sigma - \delta, \sigma) : f_3 (\zeta, s) > \frac{C}{4}|s-\sigma|^{n_1} \},\\\\
I_{5}^{\delta, \zeta, C} = \{s \in (\sigma - \delta, \sigma) : f_3 (\zeta, s) < -\frac{C}{4}|s-\sigma|^{n_1} \}.
\end{array} \right.
\end{equation}

\

For $\delta > 0 $ and $i=1,2,3,4,5$, let $J_{i}^{\delta, \zeta, C}$ be the analogous sets defined again by (\ref{Ijotas}), but replacing there $(\sigma - \delta, \sigma)$ and $n_1$ by  $(\sigma, \sigma + \delta)$ and $n_2$ respectively. Finally, for 
$\tau \in [c,d] \setminus \{ \sigma\}$ such that $\sigma \notin (\tau - \delta, \tau + \delta)$ and 
$i=1,2,3,4,5$, let $K_{i}^{\tau, \delta, \zeta, C}$ be the sets defined by (\ref{Ijotas}), but replacing there 
$(\sigma - \delta, \sigma)$ and $\frac{C}{4} |s - \sigma|^{n_1}$ by $(\tau - \delta, \tau + \delta) \cap [c,d]$ and $\frac{C}{4}$ respectively.

\begin{lemma} \label{CeW} $(i)$ For every $\eta \in S^{1}$ there exist two positive constants $\delta_{\eta}$ and $C_{\eta}$ and a neighborhood $W_{\eta}$ of $\eta$ in $S^{1}$ such that, for all $\zeta \in W_{\eta}$ and all $j=1,2, ...,5$, the sets 
$I_{j}^{\delta_{\eta}, \zeta, C_{\eta}}$ and $J_{j}^{\delta_{\eta}, \zeta, C_{\eta}}$ have at most $n_1$ and $n_2$ connected components respectively.

$(ii)$ For every $\tau \in [c,d] \setminus \{ \sigma \}$ and $\eta \in S^{1}$ there exist two positive constants $\delta_{\eta, \tau}$ and 
$C_{\eta, \tau}$ and a neighborhood $W_{\eta, \tau}$ of $\eta$ in $S^{1}$ such that, for all $\zeta \in W_{\eta, \tau}$ and all $j=1, ...,5$, either $K_{j}^{\tau, \delta_{\eta, \tau}, \, \zeta, \, C_{\eta, \tau}} = \emptyset$ or 
$K_{j}^{\tau, \delta_{\eta, \tau}, \, \zeta, \, C_{\eta, \tau}} = (\tau - \delta_{\eta, \tau}, \tau + \delta_{\eta, \tau}) \cap [c,d]$.
\end{lemma}

\begin{proof}
To see $(i)$, we will show, for an appropriate positive constant $C$, that the sets where the $f_i$'s (see (\ref{f123})) are equal to 
$\pm\frac{C}{4} |s -\sigma|^{n_j}$ ($j=1,2$) are boundedly finite, so the boundedness of the number of connected components 
will follow from this. For them, we proceed as follows. Given $f \in C^{\infty}(a,b)$, let $d_f$ be defined by 
\[
d_f = \left\{\begin{array}{cc}
              \min \{ l \geq 0 : f^{(l)}(\sigma) \neq 0 \} & \text{if} \,\, f^{(k)}(\sigma) \neq 0 \,\, \text{for some} \,\, k \geq 0, \\\\
							+\infty  & \,\, \text{otherwise}
\end{array} \right..
\]
For $i=1,2,3$, let $f_i : S^{1} \times (a,b) \to \mathbb{R}$ be the functions given by (\ref{f123}) 
and, for $\eta \in S^{1}$ fixed, let $k_i = d_{f_{i}(\eta, \, \cdot \,)}$. By Remark \ref{lower bound}, $(i)$, we have that
$d_{(f_{1} - f_{2} + f_{3})(\eta, \, \cdot \,)} \leq n_1$ and thus $k_i \leq n_1$ for some $i=1,2,3$. Let 
$A = \{ i \in \{ 1,2,3 \} : k_i \leq n_1 \}$ and for $\eta \in S^{1}$ let 
$C_{\eta} := \min \left\{\left| \frac{1}{k_{i}!} \partial^{k_i}_{s} f_{i}(\eta, \sigma) \right|: i \in A \right\}$, where $\partial_s$ denotes the partial derivative respect to $s$. For 
$(\zeta, s) \in S^{1} \times (a,b)$ let
\[
F_{i}(\zeta, s) = f_{i}(\zeta, s) - \frac{C_{\eta}}{4} (s-\sigma)^{n_1}, \,\,\,\,\,\, i=1,2,3.
\]
If $i \in A$, we have that $\partial^{k_i}_{s} F_{i}(\eta, \sigma) \neq 0$ and so there exists $\delta_{\eta, i} > 0$ and a neighborhood 
$W_{\eta, i}$ of $\eta$ in $S^{1}$ such that $\partial^{k_i}_{s} F_{i}(\zeta, s) \neq 0$ for all 
$(\zeta, s) \in W_{\eta, i} \times (\sigma - \delta_{\eta, i}, \sigma)$. Then, by the Rolle's Theorem for $\zeta \in W_{\eta, i}$ and 
$i \in A$, the equation $f_{i}(\zeta, s) = \frac{C_{\eta}}{4} (s-\sigma)^{n_1}$ has at most $k_i$ roots in $(\sigma - \delta_{\eta, i}, \sigma)$ with $k_i \leq n_1$. Now, if $i \notin A$, then  $\partial^{n_1}_{s} F_{i}(\eta, \sigma) \neq 0$ and so, proceeding as above, we find also in this case $\delta_{\eta, i} > 0$ and $W_{\eta, i}$ such that for $\zeta \in W_{\eta, i}$, 
$f_{i}(\zeta, s)= \frac{C_{\eta}}{4} (s-\sigma)^{n_1}$ has at most $n_1$ roots in $(\sigma - \delta_{\eta, i}, \sigma)$. Similarly, by considering $\widetilde{F}_{i}(\zeta,s)=f_{i} (\zeta,s) + \frac{C_{\eta}}{4}(s-\sigma)^{n_1}$ instead of $F_{i}$, and diminishing 
$W_{\eta, i}$ and $\delta_{\eta, i}$ if necessary, we obtain that for $\zeta \in W_{\eta, i}$,
$f_{i}(\zeta, s)= -\frac{C_{\eta}}{4} (s-\sigma)^{n_1}$ has at most $n_1$ roots in $(\sigma - \delta_{\eta, i}, \sigma)$.
From these facts it is clear that the assertion in $(i)$ holds for the sets $I_{j}^{\delta_{\eta}, \zeta, C_{\eta}}$ if we take
$\delta_{\eta} = \min \{ \delta_{\eta, i} : i=1,2,3 \}$ and $W_{\eta} = \bigcap_{i=1}^{3} W_{\eta, i}$. The proof of the statement for the sets $J_{j}^{\delta_{\eta}, \zeta, C_{\eta}}$ is similar.

To see $(ii)$, let $ \tau \in [c,d] \setminus \{ \sigma \}$ and let $f_i$ be the functions defined above. For a constant $\tilde{D} > 0$ as in Lemma \ref{eigen functions} (\textit{iv}), from Remark \ref{lower bound} $(iii)$, we have that
\[
|f_1(\zeta, \tau) - f_2(\zeta, \tau) + f_3(\zeta, \tau)| \geq \tilde{D},
\]
for all $\zeta \in S^{1}$. For $\eta \in S^{1}$, we define $A=\{ i \in \{1,2,3\}: f_{i}(\eta, \tau) \neq 0\}$ and
$C_{\eta, \tau} := \min \{ \{\tilde{D} \} \cup \{|f_{i}(\eta, \tau)|: i \in A\}\}$. For $(\zeta, s) \in S^{1} \times (a,b)$ and $i=1,2,3$, let
$F_{i}(\zeta, s)=f_{i}(\zeta, s)-\frac{C_{\eta, \tau}}{4}$ and $\widetilde{F}_{i}(\zeta, s)=f_{i}(\zeta, s)+\frac{C_{\eta, \tau}}{4}$.
Thus $F_{i}(\eta, \tau) \neq 0$ and $\widetilde{F}_{i}(\eta, \tau) \neq 0$ for $i=1,2,3$ and so there exist $\delta_{\eta, \tau, i}$  
and a neighborhood $W_{\eta, \tau, i}$ of $\eta$ in $S^{1}$ such that $F_{i}(\zeta, s) \neq 0$ and $\widetilde{F}_{i}(\zeta, s) \neq 0$
for all $\zeta \in W_{\eta, \tau, i}$ and all $s \in (\tau - \delta_{\eta, \tau, i}, \tau + \delta_{\eta, \tau, i})\cap [c,d]$. We set
$W_{\eta, \tau} = \bigcap_{i=1}^{3} W_{\eta, \tau, i}$ and $\delta_{\eta, \tau} = \min \{ \delta_{\eta, \tau, i} : i=1,2,3 \}$.
Since for $\zeta \in W_{\eta, \tau}$ and $i=1,2,3$, the equations $f_{i}(\zeta, s) = \pm \frac{C_{\eta, \tau}}{4}$ have not roots in 
$(\tau - \delta_{\eta, \tau}, \tau + \delta_{\eta, \tau}) \cap [c,d]$ it follows, for every $j$, that 
$K_{j}^{\tau, \delta_{\eta, \tau}, \zeta, C_{\eta, \tau}} = \emptyset$ or $K_{j}^{\tau, \delta_{\eta, \tau}, \zeta, C_{\eta, \tau}} =
(\tau - \delta_{\eta, \tau}, \tau + \delta_{\eta, \tau})\cap [c,d]$. This completes the proof.
\end{proof}

\begin{remark} \label{CetaD}
We observe that if $D$ is a positive constant as in Lemma \ref{eigen functions} (\textit{iii}), then for each $\eta \in S^{1}$ we can find a positive constant $C_{\eta} \leq D$ for which Lemma \ref{CeW}, (\textit{i}), still holds. 
\end{remark}

\begin{lemma} \label{ineqs1}
Let $f \in C^{1}([\theta_1, \theta_2])$, $\gamma > 1$ and let $I$ be a closed interval such that $I \subseteq [\theta_1, \theta_2]$. If one of the following two conditions holds
\begin{equation} \label{fprime1}
(f'(s))^{2} \geq A (s - \theta_1)^{\alpha} \,\,\,\, \text{for all} \,\,\, s \in I,
\end{equation}
\begin{equation} \label{fprime2}
(f'(s))^{2} \geq A (\theta_2 - s)^{\alpha} \,\,\,\, \text{for all} \,\,\, s \in I,
\end{equation}
for positive constants $A$ and $\alpha$ with $\alpha < 2\gamma - 2$, then there exists a constant $C > 0$ depending only on $A$,
$\alpha$, $\theta_1$, $\theta_2$ and $\gamma$ such that
\begin{equation} \label{int1}
\int_{I} |f(s)|^{-\frac{1}{\gamma}} ds \leq C.
\end{equation}
\end{lemma}

\begin{proof}
Let $I=[a,b]$. Assume that (\ref{fprime1}) holds. Since $f'$ is continuous we have that either $(i)$ $f'(s) \geq 0$ for all
$s \in I$ or $(ii)$ $f'(s) \leq 0$ for all $s \in I$. Consider the case $(i)$. It is clear that $f'(s) > 0$ for all $s \in (a,b)$, so $f$ is strictly increasing on $[a,b]$. Thus only one of the following three cases occurs

${\bf a)}$ $f$ is positive everywhere in $(a,b]$ with $f(a) \geq 0$,

or

${\bf b)}$ $f$ is negative everywhere in $[a,b)$ with $f(b) \leq 0$,

or

${\bf c)}$ $f$ vanishes at exactly a point $s_0 \in (a,b)$.
\\
If ${\bf a})$ holds, then for every $s \in I$
\[
|f(s)| \geq \int_{a}^{s} f'(t) \, dt \geq \frac{A^{1/2}}{1+\frac{\alpha}{2}} \left( (s- \theta_1)^{1+\frac{\alpha}{2}} - (a-\theta_1)^{1+\frac{\alpha}{2}}  \right) \geq \frac{A^{1/2}}{1+\frac{\alpha}{2}} \left( s-a  \right)^{1+\frac{\alpha}{2}}
\]
which, taking into account that $\alpha < 2\gamma - 2$, implies (\ref{int1}) with the constant $C >0$ required.
\\
If ${\bf b})$ holds, then for every $s \in I$
\[
|f(s)| \geq \int_{s}^{b} f'(t) \, dt \geq \frac{A^{1/2}}{1+\frac{\alpha}{2}} \left( (b- \theta_1)^{1+\frac{\alpha}{2}} - (s-\theta_1)^{1+\frac{\alpha}{2}}  \right) \geq \frac{A^{1/2}}{1+\frac{\alpha}{2}} \left( b-s  \right)^{1+\frac{\alpha}{2}}
\]
which implies (\ref{int1}) with the constant $C >0$ required since $\alpha < 2\gamma - 2$.
\\
If ${\bf c})$ holds, we have for $a \leq s \leq s_0$
\[
|f(s)| \geq \int_{s}^{s_0} f'(t) \, dt \geq \frac{A}{1+\frac{\alpha}{2}} \left( (s_0 - \theta_1)^{1+\frac{\alpha}{2}} - (s-\theta_1)^{1+\frac{\alpha}{2}}  \right) \geq \frac{A}{1+\frac{\alpha}{2}} \left( s_0 - s \right)^{1+\frac{\alpha}{2}}
\]
and thus $\int_{a}^{s_0} |f(s)|^{-\frac{1}{\gamma}} \, ds \leq C$, where $C$ is the constant required. Similarly, one can estimate 
$\int_{s_0}^{b} |f(s)|^{-\frac{1}{\gamma}} \, ds$. So (\ref{int1}) holds in the case ${\bf c})$. The case $(ii)$ reduces to $(i)$ by considering there $-f$ instead of $f$. Finally, the proof of that (\ref{fprime2}) $\implies$ (\ref{int1}) is similar.
\end{proof}

\begin{lemma} \label{ineqs2}
Let $f \in C^{2}([\theta_1, \theta_2])$, $\gamma > 1$ and let $I$ be a closed interval such that $I \subseteq [\theta_1, \theta_2]$. If one of the following two conditions holds
\begin{equation} \label{fprime3}
|f'(s) f''(s)| \geq A (s - \theta_1)^{\alpha} \,\,\,\, \text{for all} \,\,\, s \in I,
\end{equation}
\begin{equation} \label{fprime4}
|f'(s) f''(s)| \geq A (\theta_2 - s)^{\alpha} \,\,\,\, \text{for all} \,\,\, s \in I,
\end{equation}
for positive constants $A$ and $\alpha$ with $\alpha < 2\gamma - 3$, then ({\ref{int1}}) holds for some $C > 0$ depending only on $A$,
$\alpha$, $\theta_1$, $\theta_2$ and $\gamma$.
\end{lemma}

\begin{proof}
Let $I=[a,b]$. If (\ref{fprime3}) holds, we have that either $(i)$ $\left((f')^{2}\right)'(t) \geq 2A(t-\theta_1)^{\alpha}$ for all 
$t \in [a,b]$ or $(ii)$ $\left((f')^{2}\right)'(t) \leq -2A(t-\theta_1)^{\alpha}$ for all $t \in [a,b]$. If $(i)$ holds, by integrating
on $[a,s]$, we obtain
\[
(f')^{2}(s) \geq (f')^{2}(a) + c_1 \left( (s-\theta_1)^{1+\alpha} - (a-\theta_1)^{1+\alpha} \right) \geq c_1(s-a)^{1+\alpha} \,\,\,\,\,\,  
\text{for all} \, s \in I,
\]
where $c_1$ is a positive constant depending only on $A$ and $\alpha$. If $(ii)$ holds, by integrating on $[s,b]$, we get similarly as 
above that $(f')^{2}(s) \geq c_2 (b-s)^{1+\alpha}$ for all $s \in I$, with $c_2=c_2(A, \alpha)$. Since $1+\alpha < 2 \gamma -2$, the 
lemma follows, in both cases $(i)$ and $(ii)$, from Lemma \ref{ineqs1}. Finally, the case when (\ref{fprime4}) holds is similar.
\end{proof}

\begin{lemma} \label{ineqs3}
Let $f \in C^{2}([\theta_1, \theta_2])$, $\gamma > 1$ and let $I$ be a closed interval such that $I \subseteq [\theta_1, \theta_2]$. If one of the following two conditions holds
\begin{equation} \label{fprime5}
f(s) f''(s) \geq A (s - \theta_1)^{\alpha} \,\,\,\, \text{for all} \,\,\, s \in I,
\end{equation}
\begin{equation} \label{fprime6}
f(s) f''(s) \geq A (\theta_2 - s)^{\alpha} \,\,\,\, \text{for all} \,\,\, s \in I,
\end{equation}
for positive constants $A$ and $\alpha$ with $\alpha < 2\gamma - 2$, then ({\ref{int1}}) holds for some $C > 0$ depending only on $A$,
$\alpha$, $\theta_1$, $\theta_2$ and $\gamma$.
\end{lemma}

\begin{proof}
Let $I=[a,b]$. We give the proof only for the case when (\ref{fprime5}) holds, since the proof when (\ref{fprime6}) holds is similar. We assume that (\ref{fprime5}) occurs, in this case either 

$(i)$ $f \geq 0$ and $f'' \geq 0$ on $I$,

or

$(ii)$ $f \leq 0$ and $f'' \leq 0$ on $I$.
\\
If $(i)$ holds, let $s_1$ be the point where the minimum of $f$ on $I$ is achieved. If $s_1 = a$, then $f(a) \geq 0$ and $f'(a) \geq 0$. 
From (\ref{fprime5}), to integrate by parts, we get
\[
\frac{1}{2} (f^{2})'(t) \geq f(t) f'(t) - f(a) f'(a) - \int_{a}^{t} (f')^{2}(r) \, dr = \int_{a}^{t} f(r) f''(r) \, dr 
\]
\[
\geq c_1 \left( (t-\theta_1)^{1+\alpha} -
(a-\theta_1)^{1+\alpha} \right) \geq c_1 (t-a)^{1+\alpha},
\]
for a positive constant $c_1=c_1(A, \alpha)$ and for all $t \in I$. A new integration on $[a,s]$ gives, for some positive constant 
$c_2 = c_2(A, \alpha)$,
\begin{equation} \label{f2}
f^{2}(s) \geq c_2 (s-a)^{2+\alpha}, \,\,\,\, \text{for all} \, s \in I.
\end{equation}
Since $\alpha < 2\gamma - 2$, (\ref{f2}) gives (\ref{int1}) with a constant $C$ as required. Now, if $s_1 = b$ we have that $f(b) > 0$ and
$f'(r)\leq f'(b) \leq 0$ for every $r \in I$. Using (\ref{fprime5}), to integrate by parts on $[t,b]$, we get
$-(f^{2})'(t) \geq c_3 (b-t)^{1+\alpha}$ for a positive constant $c_3 = c_3(A,\alpha)$ and all $t \in I$. Then, the proof in this case follows as above, but integrating on $[s,b]$. Now, we consider the case $a < s_1 < b$. This case follows to apply the two cases above on the intervals $[s_1, b]$ and $[a, s_1]$ respectively.

Finally, the case $(ii)$ reduces to $(i)$ by considering there $-f$ instead of $f$.
\end{proof}

\begin{lemma} \label{ineqs4}
Let $f \in C^{2}([\theta_1, \theta_2])$, $\gamma > 1$ and let $I$ be a closed interval such that $I \subseteq [\theta_1, \theta_2]$. If one of the following two conditions holds
\begin{equation} \label{fprime7}
f(s) f''(s) \leq -A (s - \theta_1)^{\alpha} \,\,\,\, \text{for all} \,\,\, s \in I,
\end{equation}
\begin{equation} \label{fprime8}
f(s) f''(s) \leq -A (\theta_2 - s)^{\alpha} \,\,\,\, \text{for all} \,\,\, s \in I,
\end{equation}
for positive constants $A$ and $\alpha$ with $\alpha < 2\gamma - 2$, then ({\ref{int1}}) holds for some $C > 0$ depending only on $A$,
$\alpha$, $\theta_1$, $\theta_2$ and $\gamma$.
\end{lemma}

\begin{proof}
The proof of the case when (\ref{fprime8}) holds is analogous to that of when (\ref{fprime7}) occurs, so we will prove only that 
(\ref{fprime7}) $\implies$ (\ref{int1}).
Let $I=[a,b]$. From (\ref{fprime7}) we have either $(i)$ $f(s) > 0$ and $f''(s) < 0$ for all $s \in I \setminus \{ \theta_1 \}$ or
$(i)$ $f(s) < 0$ and $f''(s) > 0$ for all $s \in I \setminus \{ \theta_1 \}$. Consider the case $(i)$. Let $u$ and $\Phi$ be the functions defined on $[a,b]$ by
\[
u(s) = \sin \left( \frac{(s-a)\pi}{b-a} \right),
\]
and
\[
\Phi(s) = A (b-a)^{\alpha} \pi^{-\alpha} u^{\alpha}(s).
\]
From (\ref{fprime7}) we have that $f(a) \geq 0$, $f(b) > 0$ and $-f''(s) \geq \Phi(s)/f(s)$ for all $s \in (a,b]$. Let $\widetilde{u}$ and 
$\widetilde{\Phi}$ be the functions defined on $[a,b]$ by
\[
\widetilde{u}(s) = u^{1+\frac{\alpha}{2}}(s),
\]
and
\[
\widetilde{\Phi}(s)= \frac{\pi^{2}(1+\frac{\alpha}{2})}{(b-a)^{2}} \left[  u^{2+\alpha}(s) - \frac{\alpha}{2} 
\cos^{2}\left( \frac{(s-a)\pi}{b-a} \right) u^{\alpha}(s) \right].
\]
A computation shows that $\widetilde{u}$ satisfies $-\widetilde{u}'' = \widetilde{\Phi}/\widetilde{u}$ on $(a,b)$ and 
$\widetilde{u}(a) = \widetilde{u}(b) = 0$. Also, $\widetilde{\Phi} \leq \beta \Phi$ on $(a,b)$, with 
$\beta = A^{-1} (b-a)^{-(2+\alpha)} \pi^{2+\alpha} (1+\frac{\alpha}{2})^{2}$. We claim that
\begin{equation} \label{uf1}
\widetilde{u} \leq \beta^{1/2} f \,\,\,\,\,\,\,\, \text{on} \,\,\, (a,b).
\end{equation}
Indeed, we have $-\widetilde{u}'' \leq \beta \Phi/\widetilde{u}$ and $-(\beta^{1/2}f)'' \geq \beta \Phi (\beta^{1/2}f)^{-1}$ on $(a,b)$. Then
\begin{equation} \label{uf2}
-(\beta^{1/2}f-\widetilde{u})'' \geq -\beta \Phi \frac{\beta^{1/2}f-\widetilde{u}}{\beta^{1/2} f \widetilde{u}} \,\,\,\,\,\,\,\, on \,\,\, 
(a,b).
\end{equation}
Let $I^{\ast} = \{ s \in (a,b) : \beta^{1/2}f < \widetilde{u} \}$ and suppose that $I^{\ast} \neq \emptyset$. Let $J$ be a connected 
component of $I^{\ast}$. Since $\beta^{1/2}f-\widetilde{u}$ is continuous on $I^{\ast}$ and $\beta^{1/2}f-\widetilde{u}$ is nonnegative 
at $a$ and at $b$, it follows that $\beta^{1/2}f-\widetilde{u}$ vanishes at $\partial J$, but $\Phi > 0$ on $(a,b)$, so (\ref{uf2}) 
implies that $\beta^{1/2}f-\widetilde{u}$ is strictly concave on $J$. Being $\beta^{1/2}f-\widetilde{u}$ negative on $J$, we get a contradiction. Thus (\ref{uf1}) holds. Then
\[
\int_{I} |f(s)|^{-\frac{1}{\gamma}} \, ds \leq \beta^{\frac{1}{2\gamma}} \int_{I} [\widetilde{u}(s)]^{-\frac{1}{\gamma}} \, ds
= \beta^{\frac{1}{2\gamma}} (b-a)\pi^{-1} \int_{0}^{\pi} (\sin(\theta))^{-\frac{2+\alpha}{2\gamma}} \, d\theta < \infty
\]
since $\alpha < 2\gamma -2$. Finally, the case $(ii)$ follows from the previous one applied to $-f$.
\end{proof}

We recall the definition of our singular measure $\mu$ in $\mathbb{R}^{4}$ which is supported on the graph of a real analytic function 
$\varphi : V^{a,b} \to \mathbb{R}^{2}$ satisfying H1 - H3. We assume, in addition to H1 - H3, that

\

H4) $\max \{ n_1, n_2 \} < \displaystyle{\frac{2m}{\alpha_1 + \alpha_2}} - 3$,

\
\\
where $n_1$ and $n_2$ are the two positive integers appearing in Lemma \ref{eigen functions}, (\textit{iii}). The number 
$\max \{ n_1, n_2 \}$ is a kind of measure of the degeneracy of ellipticity along the curve $\{ t \bullet (1, \sigma) : t >0 \}$ 
(see hypothesis H3). Now, we consider the set $V_{1}^{c,d} \subset V^{a,b}$ given by (\ref{set Vcd1}), where $\sigma \in [c,d]$. Then, for every Borel set $E$ of $\mathbb{R}^{4}$, we define the singular measure $\mu$ by
\begin{equation} \label{mu2}
\mu(E) = \int_{V_{1}^{c,d}} \chi_{E}(x, \varphi(x)) \, dx.
\end{equation}

In the following theorem, assuming H1 - H4, we give a decay estimate for the Fourier transform of $\mu$.

\begin{theorem} \label{muhat estimate}
Let $\mu$ be the singular measure given by (\ref{mu2}). Then there exist a positive 
constant $C$ such that
\begin{equation} \label{muhat2}
\left|  \widehat{\mu}(\xi', \xi'') \right| \leq C \, |\xi''|^{-\frac{\alpha_1+\alpha_2}{m}}, \,\,\,\,\,\, \text{for all} \,\,
\xi' \in \mathbb{R}^{2} \,\, \text{and all} \,\, \xi'' \in \mathbb{R}^{2} \setminus \{ {\bf 0} \}.
\end{equation}  
\end{theorem}

\begin{proof}
We consider the case $\sigma \in (c,d)$; if $c = \sigma$ we work with the second inclusion in (\ref{inc}) below, now if $d = \sigma$, then we work with the first. From (\ref{muhat}) and 
Lemma \ref{int osc}, we have for universal constant $C > 0$ that
\[
\left|  \widehat{\mu}(\xi', \xi'') \right| \leq C \, |\xi''|^{-\frac{\alpha_1+\alpha_2}{m}} 
\int_{c}^{d} |G(\zeta, s)|^{-\frac{\alpha_1+\alpha_2}{m}} \, ds, \,\,\,\,\,\, \text{for all} \,\,
(\xi', \xi'') \in \mathbb{R}^{2} \times \mathbb{R}^{2} \setminus \{ {\bf 0} \},
\]
where $\zeta = \frac{\xi''}{|\xi''|}$. Given $D > 0$ as in Lemma \ref{eigen functions} (\textit{iii}), from 
Remark \ref{lower bound}, $(i)$ and $(ii)$, for $\delta > 0$ and small enough we have
\begin{equation} \label{inc}
(\sigma-\delta, \sigma) \subseteq \bigcup_{j=1}^{5} I_{j}^{\delta, \, \zeta, \, D} \,\,\,\,\, \text{and} \,\,\,\,\,
(\sigma, \sigma+\delta) \subseteq \bigcup_{j=1}^{5} J_{j}^{\delta, \, \zeta, \, D} \,\,\,\, \text{for all} \,\,\,\, \zeta \in S^{1},
\end{equation}
where the sets $I_{j}^{\delta, \, \zeta, \, D}$ and $J_{j}^{\delta, \, \zeta, \, D}$ are defined by (\ref{Ijotas}). For $\eta \in S^{1}$ 
fixed, let $\delta_{\eta}$, $C_{\eta}$,  and $W_{\eta}$ be as in Lemma \ref{CeW}, $(i)$. By Remark \ref{CetaD}, we take $C_{\eta} \leq D$. From the compactness of $S^{1}$ there exists a finite set of points $\{ \eta_k \}_{k=1}^{N} \subset S^{1}$ such that $S^{1} = \bigcup_{k=1}^{N} W_{\eta_k}$. Now, we pick $\epsilon$ such that 
$0 < \epsilon < \delta$ and $0 < \epsilon < \min \{ \delta_{\eta_{k}} : 1 \leq k \leq N \}$, so
$(\sigma - \epsilon, \sigma) \cup (\sigma, \sigma + \epsilon) \subset \bigcup_{j=1}^{5} \left( I_{j}^{\epsilon, \, \zeta, \, D} \cup
J_{j}^{\epsilon, \, \zeta, \, D}  \right)$ for all $\zeta \in S^{1}$.

Given $\zeta \in S^{1}$, we have that $\zeta \in W_{\eta_k}$ for some $k=k(\zeta)$. Since $\delta_{\eta_{k(\zeta)}} \geq \epsilon$ and
$C_{\eta_{k(\zeta)}} \leq D$, it follows that $I_{j}^{\epsilon, \, \zeta, \, D} \subset I_{j}^{\delta_{\eta_{k(\zeta)}}, \, \zeta, \, 
C_{\eta_{k(\zeta)}}}$ and $J_{j}^{\epsilon, \, \zeta, \, D} \subset J_{j}^{\delta_{\eta_{k(\zeta)}}, \, \zeta, \, C_{\eta_{k(\zeta)}}}$ for every $j=1,2, ...,5$. Then
\[
(\sigma - \epsilon, \sigma) \cup (\sigma, \sigma + \epsilon) \subset \bigcup_{j=1}^{5} \left( I_{j}^{\delta_{\eta_{k(\zeta)}}, \, \zeta, \, C_{\eta_{k(\zeta)}}} \cup J_{j}^{\delta_{\eta_{k(\zeta)}}, \, \zeta, \, C_{\eta_{k(\zeta)}}}  \right).
\]
Hence 
\[
\int_{\sigma-\epsilon}^{\sigma-\epsilon} |G(\zeta,s)|^{-\frac{\alpha_1+\alpha_2}{m}} \, ds 
\]
\[
\leq
\sum_{j=1}^{5} \left[ \int_{I_{j}^{\delta_{\eta_{k(\zeta)}}, \, \zeta, \, C_{\eta_{k(\zeta)}}}} 
|G(\zeta,s)|^{-\frac{\alpha_1+\alpha_2}{m}} \, ds + \int_{J_{j}^{\delta_{\eta_{k(\zeta)}}, \, \zeta, \, C_{\eta_{k(\zeta)}}}} 
|G(\zeta,s)|^{-\frac{\alpha_1+\alpha_2}{m}} \, ds\right].
\]
By Lemma \ref{CeW}, (i), the sets $I_{j}^{\delta_{{\eta_{k(\zeta)}}}, \, \zeta, \, C_{\eta_{k(\zeta)}}}$ and 
$J_{j}^{\delta_{{\eta_{k(\zeta)}}}, \, \zeta, \, C_{\eta_{k(\zeta)}}}$ have at most $n_1$ and $n_2$ connected components respectively, so
we can estimate, uniformly on $\zeta$, the integrals on the right-side hand of this inequality by applying Lemmas \ref{ineqs1}, \ref{ineqs2}, 
\ref{ineqs3} and \ref{ineqs4} to the function $f(s) = G(\zeta,s)$ with $\gamma = \frac{m}{\alpha_1+\alpha_2}$ and $\alpha = n_1$ or 
$\alpha = n_2$ according to the case and taking into account the hypothesis H4.

Let $\tilde{D} > 0$ be as in Lemma \ref{eigen functions}, (\textit{iv}), and let $F=[c,d] \setminus (\sigma - \epsilon, \sigma + \epsilon)$, where $\epsilon$ is as above. It is clear that $F$ is compact. For $\tau \in [c,d] \setminus (\sigma - \frac{\epsilon}{2}, \sigma + \frac{\epsilon}{2})$
and $\eta \in S^{1}$, let $\delta_{\eta, \tau}$, $C_{\eta, \tau} \leq \tilde{D}$ and $W_{\eta, \tau}$ be as in 
Lemma \ref{CeW}, $(ii)$. From the compactness of $S^{1}$ there exists a finite set of points $\{ \eta_{k} \}_{k=1}^{M} \subset S^{1}$
such that $S^{1} = \bigcup_{k=1}^{M} W_{\eta_{k}, \tau}$. Now, from Remark \ref{lower bound} $(iii)$, for $\delta > 0$ and small enough we have that
$$(\tau -\delta, \tau + \delta) \cap [c,d] \subset \bigcup_{j=1}^{5} K_{j}^{\tau, \delta, \zeta, \tilde{D}}, \,\, \text{for all} \,\,
\zeta \in S^{1}.$$
Let $\delta_{\tau} = \min \{ \delta_{\eta_k, \tau} : 1 \leq k \leq M\}$.
We take $\widetilde{\epsilon} = \min \{ \delta, \delta_{\tau} \}$. Now, given $\zeta \in S^{1}$ we have
that $\zeta \in W_{\eta_k}$ for some $k=k(\zeta)$. Since $\delta_{\eta_{k(\zeta)}, \tau}=\delta_{\eta_{k}, \tau} \geq 
\widetilde{\epsilon}$ and $C_{\eta_{k(\zeta)}, \tau}=C_{\eta_{k}, \tau} \leq \tilde{D}$,  it follows that
$I_{\tau} := (\tau - \widetilde{\epsilon}, \tau + \widetilde{\epsilon}) \cap [c,d] \subset \bigcup_{j=1}^{5} 
K_{j}^{\tau, \widetilde{\epsilon}, \zeta, \tilde{D}} \subset \bigcup_{j=1}^{5} 
K_{j}^{\tau, \delta_{\eta_{k}, \tau}, \zeta, C_{\eta_{k}, \tau}}$. Thus
\[
\int_{I_{\tau}} |G(\zeta,s)|^{-\frac{\alpha_1+\alpha_2}{m}} \, ds 
\leq
\sum_{j=1}^{5} \int_{K_{j}^{\tau, \delta_{\eta_{k}, \tau}, \, \zeta, \, C_{\eta_{k}, \tau}}} 
|G(\zeta,s)|^{-\frac{\alpha_1+\alpha_2}{m}} \, ds.
\]
By Lemma \ref{CeW}, $(ii)$, each $K_{j}^{\tau, \delta_{\eta_{k}, \tau}, \, \zeta, \, C_{\eta_{k}, \tau}}$ is either empty or connected
and so the integrals on the right-side hand of this last inequality can be estimated, uniformly on $\zeta$, by using Lemmas \ref{ineqs1}, \ref{ineqs2}, \ref{ineqs3}, \ref{ineqs4} and the hypothesis H4. Since $F$ is covered by a finite number of these intervals 
$(\tau-\widetilde{\epsilon}, \tau+\widetilde{\epsilon})$, we obtain an estimate, uniformly on $\zeta$, for the integral
$\int_{F} |G(\zeta,s)|^{-\frac{\alpha_1+\alpha_2}{m}} \, ds$. This concludes the proof of the theorem.
\end{proof}

\section{A restriction theorem for the Fourier transform and $L^{p}$-improving properties of $\mu$}

Let $\Sigma = \left\{ (x, \varphi(x)) : x \in V_{1}^{c,d} \right\}$, with $\sigma \in [c,d]$ (see hyp. H3). For 
$f \in L^{1}(\mathbb{R}^{4})$, let $\mathcal{R}f = \widehat{f} \, |_{\Sigma}$ where $\widehat{f}$ denotes the usual Fourier transform. The following lemma gives a necessary condition for the $L^{p}(\mathbb{R}^{4})-L^{q}(\Sigma)$ boundedness of the restriction operator 
$\mathcal{R}$ and it is an adaptation, to our setting, of the well known Knapp's homogeneity argument.

\begin{lemma} \label{Knapp}
If there exists a positive constant $C$ such that
\begin{equation} \label{R1}
\| \mathcal{R}f \|_{L^{q}(\Sigma, d\mu)} \leq C \| f \|_{L^{p}(\mathbb{R}^{4})}, \,\,\,\,  f \in \mathcal{S}(\mathbb{R}^{4}).
\end{equation}
then $\frac{1}{q} \geq \frac{(\alpha_1 + \alpha_2 + 2m)}{(\alpha_1 + \alpha_2)} \frac{1}{p'}$.
\end{lemma}

\begin{proof}
We observe that if (\ref{R1}) holds for $f \in \mathcal{S}(\mathbb{R}^{4})$, then it holds also for 
$f \in L^{1}(\mathbb{R}^{4}) \cap L^{p}(\mathbb{R}^{4})$. Let 
$Q= \left\{ t \bullet (1,s) : s \in [s_1, s_2] \,\, \text{and} \,\, \frac{1}{2} \leq t \leq 1 \right\}$, where $c < s_1 < s_2 < d$. 
We fix $0 < \eta < 1$ such that $|\varphi(x)| < 1$ for all $x \in \eta \bullet Q$. For $\xi'$, $\xi'' \in \mathbb{R}^{2}$ and $t>0$ let 
$t \odot (\xi', \xi'') = (t \bullet \xi', t^{m} \xi'')$, and for $0 < \epsilon < 1$ we set $E_{\epsilon} = \epsilon^{-1} \odot [0,1]^{4}$ 
and $f_{\epsilon} = \chi_{E_{\epsilon}}$.

For $x=(x_1, x_2) \in Q$, a change of variable gives
\[
\widehat{f_{\epsilon}}(x, \varphi(x)) = \int_{E_{\epsilon}} e^{-i(x_1 \xi_1 + x_2 \xi_2 + \xi_3 \varphi_1(x) + \xi_4 \varphi_2(x))} \, d\xi
\]
\[
= \epsilon^{-(\alpha_1 + \alpha_2 + 2m)} \int_{[0,1]^{4}} e^{-i(\epsilon^{-\alpha_1} x_1 \xi_1 + \epsilon^{-\alpha_2} x_2 \xi_2 + 
\epsilon^{-m} \xi_3 \varphi_1(x) + \epsilon^{-m} \xi_4 \varphi_2(x))} \, d\xi,
\]
where $d\xi = d\xi_1 \cdot \cdot \cdot d\xi_4$. We define $F_{\epsilon} = (\epsilon \eta) \bullet Q$ and 
$\Sigma_{\epsilon} = \left\{ (x, \varphi(x)) : x \in F_{\epsilon} \right\}$. Then,
\[
\| \mathcal{R}f_{\epsilon} \|_{L^{q}(\Sigma, d\mu)}^{q} \geq \| \mathcal{R}f_{\epsilon} \|_{L^{q}(\Sigma_{\epsilon}, d\mu)}^{q}
\]
\[
\geq \epsilon^{-q(\alpha_1 + \alpha_2 + 2m)} \int_{F_{\epsilon}} \left| \int_{[0,1]^{4}} e^{-i(\epsilon^{-\alpha_1} x_1 \xi_1 + 
\epsilon^{-\alpha_2} x_2 \xi_2 + \epsilon^{-m} \xi_3 \varphi_1(x) + \epsilon^{-m} \xi_4 \varphi_2(x))} \, d\xi  \right|^{q} dx.
\]
Also
\[
\left| \int_{[0,1]^{4}} e^{-i(\epsilon^{-\alpha_1} x_1 \xi_1 + 
\epsilon^{-\alpha_2} x_2 \xi_2 + \epsilon^{-m} \xi_3 \varphi_1(x) + \epsilon^{-m} \xi_4 \varphi_2(x))} \, d\xi \right|
\]
\begin{equation} \label{prod integral}
=\left| \int_{0}^{1} e^{-i\epsilon^{-\alpha_1} x_1 \xi_1} d\xi_1 \right| \left| \int_{0}^{1}  e^{-i\epsilon^{-\alpha_2} x_2 \xi_2} d\xi_2  
\right| \left|\int_{0}^{1}  e^{-i\epsilon^{-m} \xi_3 \varphi_1(x)} d\xi_3  \right| \left| \int_{0}^{1}  e^{-i\epsilon^{-m} \xi_4 \varphi_2(x)} d\xi_4 \right|.
\end{equation}

For every $x \in F_{\epsilon}$ fixed, we have that $x= (x_1, x_2) = ((\epsilon \eta t)^{\alpha_1}, (\epsilon \eta t)^{\alpha_2}s)$ with $\frac{1}{2} \leq t \leq 1$ and $s_1 \leq s \leq s_2$. To estimate the two first integrals in (\ref{prod integral}) from below we choose $\eta \in (0,1)$ such that $[\eta^{\alpha_2} s_1, \eta^{\alpha_2} s_2] \subset [-1,1]$, so 
$\epsilon^{-1} \bullet x = (\epsilon^{-\alpha_1} x_1, \epsilon^{-\alpha_2} x_2) = ((\eta t)^{\alpha_1}, (\eta t)^{\alpha_2}s)) \in [-1,1] \times [-1,1]$. Thus
\[
\left| \int_{0}^{1}  e^{-i\epsilon^{-\alpha_j} \xi_j x_j} d\xi_j \right| = 
\left| \frac{e^{-i \epsilon^{-\alpha_j} x_j} - 1}{\epsilon^{-\alpha_j} x_j} \right| \geq 
\frac{\sin(\epsilon^{-\alpha_j} x_j)}{\epsilon^{-\alpha_j} x_j} \geq \sin(1), \,\,\, \text{for every} \,\, j=1,2.
\]

For $x \in F_{\epsilon}$ we also have that $|\varphi_1(\epsilon^{-1} \bullet x)| < 1$, so the third integral in (\ref{prod integral}) can be lower bounded as
\[
\left| \int_{0}^{1}  e^{-i\epsilon^{-m} \xi_3 \varphi_1(x)} d\xi_3 \right| = 
\left| \frac{e^{-i\varphi_1(\epsilon^{-1} \bullet x)} - 1}{\varphi_1(\epsilon^{-1} \bullet x)} \right| \geq 
\frac{\sin(\varphi_1(\epsilon^{-1} \bullet x))}{\varphi_1(\epsilon^{-1} \bullet x)} \geq \sin(1).
\]
The last integral in (\ref{prod integral}) can be estimated similarly. Thus
\begin{equation} \label{R2}
\| \mathcal{R}f_{\epsilon} \|_{L^{q}(\Sigma, d\mu)} \geq C \epsilon^{-(\alpha_1+\alpha_2+2m) + \frac{\alpha_1+\alpha_2}{q}},
\end{equation}
with $C > 0$ independent of $\epsilon$. On the other hand 
$\| f_{\epsilon} \|_{L^{p}(\mathbb{R}^{4})} = \epsilon^{-\frac{\alpha_1+\alpha_2+2m}{p}}$, so (\ref{R1}) and (\ref{R2}) lead to
\[
C \epsilon^{-(\alpha_1+\alpha_2+2m) + \frac{\alpha_1+\alpha_2}{q}} \leq \epsilon^{-\frac{\alpha_1+\alpha_2+2m}{p}},
\]
this implies, taking $\epsilon$ small enough, that  
$\frac{1}{q} \geq \frac{(\alpha_1 + \alpha_2 + 2m)}{(\alpha_1 + \alpha_2)} \frac{1}{p'}$.
\end{proof}

P. Tomas in \cite{Tomas} pointed out the following result (see also \cite{Str}, Lemma 1).

\begin{lemma} \label{muhat conv estim}
If for some $1 \leq p < 2$, the inequality 
$\| \widehat{\mu} \ast f \|_{L^{p'}(\mathbb{R}^{4})} \leq C_{p}^{2} \| f \|_{L^{p}(\mathbb{R}^{4})}$ holds for all 
$f \in \mathcal{S}(\mathbb{R}^{4})$, then for that $p$
$$\| \mathcal{R}f \|_{L^{2}(\Sigma, d\mu)} \leq C_p \|f \|_{L^{p}(\mathbb{R}^{4})},$$ 
for all $f \in \mathcal{S}(\mathbb{R}^{4})$.
\end{lemma}

We are now in position to prove a sharp estimate for the restriction operator $\mathcal{R}$. For them, we embed the operator 
$T_{\widehat{\mu}}$ defined by $T_{\widehat{\mu}} f = \widehat{\mu} \ast f$ in an analytic family $\{ T_z \}$ of operators on the strip 
$-\frac{\alpha_1 + \alpha_2}{m} \leq \Re(z) \leq 2$, and then we apply the complex interpolation theorem followed by 
Lemma \ref{muhat conv estim} with $p= \frac{2(\alpha_1+\alpha_2+2m)}{\alpha_1+\alpha_2+4m}$.

\begin{theorem} \label{Restrict thm} There exists a positive constant $C$ such that
\begin{equation} \label{Restrict R}
\| \mathcal{R}f \|_{L^{2}(\Sigma, d\mu)} \leq C \| f \|_{L^{p}(\mathbb{R}^{4})}, \,\,\,\,  f \in \mathcal{S}(\mathbb{R}^{4})
\end{equation}
if and only if $\frac{\alpha_1+\alpha_2+4m}{2(\alpha_1+\alpha_2+2m)} \leq \frac{1}{p} \leq 1$.
\end{theorem}

\begin{proof} To prove the statement of the theorem we consider the family $\{ |s|^{z-2} \}$ of functions initially defined when 
$\Re(z) > 0$ and $s \in \mathbb{R}^{2} \setminus\{ 0 \}$. This family of functions can be extended, in the $z$ variable, to an analytic family of distributions on $\mathbb{C} \setminus \{ -2k : k \in \mathbb{N} \cup \{0 \} \}$. By abuse of notation, we denote this extension 
by $|s|^{z-2}$. The family 
$\{ |s|^{z-2} \}$ have simple poles in $z = -2k$ for $k \in \mathbb{N} \cup \{0 \}$. Since the meromorphic continuation of the function 
$\Gamma \left( \frac{z}{2}\right)$ (we keep the notation for his continuation) has simple poles at the same points (i.e. $z = -2k$), the family $\{ I_z \}$ of distributions defined by
\begin{equation}
I_{z}(s)=\frac{2^{-\frac{z}{2}}}{\pi \Gamma \left( \frac{z}{2}\right) }   
| s | ^{z-2}  \label{iz}
\end{equation}
results in an entire family of distributions (see pp. 71-74 in \cite{Gelfand}).

From this construction, and by taking the ratios of the corresponding residues at $z=0$, we have $I_{0} = \delta$,  where $\delta $ is the Dirac distribution at the origin on $\mathbb{R}^{2}$ (see equation (9), pp. 74 in \cite{Gelfand}), also $\widehat{I_{z}}= 4\pi I_{2-z}$ 
(see equation ($2'$), pp. 194 in \cite{Gelfand}).

For $z \in \mathbb{C}$, we also define $J_{z}$ as the distribution on $\mathbb{R}^{4}$ given by the tensor product
\begin{equation} \label{Jz}
J_{z} = \delta \otimes I_{z},
\end{equation}
where $\delta$ is the Dirac distribution at the origin on $\mathbb{R}^{2}$ and $I_z$ is given by (\ref{iz}).

Let $\{ T_z \}$ be the analytic family of operators on the strip $-\frac{\alpha_1+\alpha_2}{m} \leq \Re(z) \leq 2$, given by
\[
T_z f = \widehat{(J_z \ast \mu)} \ast f, \,\,\, \text{for} \,\, f \in \mathcal{S}(\mathbb{R}^{4}).
\]
It is clear that $T_0 = T_{\widehat{\mu}}$. For $\Re(z) = -\frac{\alpha_1+\alpha_2}{m}$, we have
\[
\left| \widehat{J_z}(\xi) \right| \leq \left| \frac{2^{\frac{z}{2}+1}}{\Gamma\left(1-\frac{z}{2}\right)} \right|
\left| \xi'' \right|^{\frac{\alpha_1+\alpha_2}{m}}, \,\,\, \text{for all} \,\, 
\xi = (\xi', \xi'') \in \mathbb{R}^{2} \times \mathbb{R}^{2}.
\]
Thus, Young's inequality and Theorem \ref{muhat estimate} give
\begin{equation} \label{estimTz1}
\| T_z f\|_{L^{\infty}(\mathbb{R}^{4})} = \| (\widehat{J_z} \widehat{\mu}) \ast f \|_{L^{\infty}(\mathbb{R}^{4})} \leq C
\left| \frac{2^{\frac{z}{2}+1}}{\Gamma\left(1-\frac{z}{2}\right)} \right| \| f \|_{L^{1}(\mathbb{R}^{4})}, \,\,\, \text{for} \,\,
\Re(z) = -\frac{\alpha_1+\alpha_2}{m},
\end{equation}
where $C$ does not depend on $z$. On the other hand, one also can see that $(J_z \ast \mu)(x',x'') = \frac{2^{-\frac{z}{2}}}{\pi \Gamma \left( \frac{z}{2} \right)} |x'' - \varphi(x')|^{z-2}$ for $x', x'' \in \mathbb{R}^{2}$. Then, for $\Re(z) = 2$ we have
\[
\| J_z \ast \mu \|_{L^{\infty}(\mathbb{R}^{4})} \leq \left| \frac{2^{-\frac{z}{2}}}{\pi \Gamma\left(\frac{z}{2}\right)} \right|.
\]
So 
\begin{equation} \label{estimTz2}
\| T_z f \|_{L^{2}(\mathbb{R}^{4})} = \left\| \widehat{T_z f} \right\|_{L^{2}(\mathbb{R}^{4})} \leq 
\left| \frac{2^{-\frac{z}{2}}}{\pi \Gamma\left(\frac{z}{2}\right)} \right| \left\| \widehat{f} \, \right\|_{L^{2}(\mathbb{R}^{4})} = 
\left| \frac{2^{-\frac{z}{2}}}{\pi \Gamma\left(\frac{z}{2}\right)} \right| \| f \|_{L^{2}(\mathbb{R}^{4})}, \,\,\, \text{for} 
\,\, \Re(z) = 2. 
\end{equation}
Now, we check that the family $\{ T_z \}$ satisfies the hypotheses of the Stein's complex interpolation theorem 
(see \cite{SteinWeiss}, pp. 205) on the strip $-\frac{\alpha_1+\alpha_2}{m} \leq \Re(z) \leq 2$. For them, let
\[
M_{-\frac{\alpha_1+\alpha_2}{m}}(y) := 
\left| \frac{2^{-\frac{\alpha_1+\alpha_2}{2m}+\frac{iy}{2}+1}}{\Gamma\left(1 +\frac{\alpha_1+\alpha_2}{2m}-\frac{iy}{2} \right)} \right|
\,\,\,\,\, \text{and} \,\,\,\,\,
M_{2}(y) := \left| \frac{2^{-\frac{2+iy}{2}}}{\pi \Gamma\left(\frac{2+iy}{2}\right)} \right|. 
\]
We recall the Stirling's formula (see e.g. \cite{Stein4}). This states that 
\[
\Gamma(z) \sim \sqrt{2 \pi} z^{z-\frac{1}{2}} e^{-z}, \,\,\,\, \text{as} \,\, |z| \rightarrow +\infty \,\,\,\, \text{and} \,\, 
|\arg(z)| \leq \frac{\pi}{2}.
\]
Then, for every $a > 0$ fixed, with the aid of the Stirling's formula, it is easy to check that
\begin{equation} \label{Stirling}
\sup_{-\infty < y < +\infty} e^{-a|y|} \log M_{-\frac{\alpha_1+\alpha_2}{m}}(y) < +\infty \,\,\,\,\,\, \text{and} \,\,\,\,
\sup_{-\infty < y < +\infty} e^{-a|y|} \log M_{2}(y) < +\infty.
\end{equation}
From (\ref{estimTz1}), (\ref{estimTz2})  and (\ref{Stirling}) taking there $0 < a < \frac{\pi}{2 + \frac{\alpha_1 + \alpha_2}{m}}$ follow that the family $\{ T_z \}$ satisfies the hypotheses of the complex interpolation theorem on the strip 
$-\frac{\alpha_1+\alpha_2}{m} \leq \Re(z) \leq 2$ and then, since $-\frac{\alpha_1+\alpha_2}{m} t+ 2 (1-t)=0$ for $t= \frac{2m}{\alpha_1+\alpha_2+2m}$, $T_0 = T_{\widehat{\mu}}$ is a bounded operator from $L^{p}(\mathbb{R}^{4})$ into $L^{p'}(\mathbb{R}^{4})$ for 
$\frac{1}{p} = t + \frac{1-t}{2} = \frac{\alpha_1+\alpha_2+ 4m}{2(\alpha_1+\alpha_2+2m)}$. This gives (\ref{Restrict R}) for that $p$ (see Lemma \ref{muhat conv estim}). Finally, the global estimate 
$\| \widehat{f} \|_{L^{\infty}(\mathbb{R}^{4})} \leq \| f\|_{L^{1}(\mathbb{R}^{4})}$ implies that
$\mathcal{R} : L^{1}(\mathbb{R}^{4}) \to L^{2}(\Sigma, \mu)$ is bounded, so the theorem follows from the Riesz-Thorin Convexity Theorem and 
Lemma \ref{Knapp}.
\end{proof}

\begin{corollary}
If $(\frac{1}{p}, \frac{1}{q})$ belongs to the closed quadrilateral with vertices $(1,0)$, $(1,1)$, 
$(\frac{\alpha_1+\alpha_2+4m}{2(\alpha_1+\alpha_2+2m)}, 1)$ and $(\frac{\alpha_1+\alpha_2+4m}{2(\alpha_1+\alpha_2+2m)}, \frac{1}{2})$, then
the restriction operator $\mathcal{R}$ is bounded from $L^{p}(\mathbb{R}^{4})$ into $L^{q}(\Sigma, d\mu)$.
\end{corollary}

\begin{proof}
The global estimate $\| \widehat{f} \|_{L^{\infty}(\mathbb{R}^{4})} \leq \| f\|_{L^{1}(\mathbb{R}^{4})}$ gives, for every 
$1 \leq q \leq \infty$, the $L^{1}(\mathbb{R}^{4}) \rightarrow L^{q}(\Sigma, d\mu)$ bound for $\mathcal{R}$. To apply Holder's inequality with $p=2$ followed by Theorem \ref{Restrict thm} we obtain the 
$L^{\frac{2(\alpha_1+\alpha_2+2m)}{\alpha_1+\alpha_2+4m}}(\mathbb{R}^{4}) \rightarrow L^{1}(\Sigma, d\mu)$ bound for $\mathcal{R}$.
Finally, the corollary follows from the Riesz-Thorin Convexity Theorem.
\end{proof}

Let $\mu$ be a Borel measure and let $T_{\mu}$ be the convolution operator by $\mu$, defined by $T_{\mu}f = \mu \ast f$, 
$f \in \mathcal{S}(\mathbb{R}^{4})$, and let $E_{\mu}$ be the set of all pairs 
$\left( \frac{1}{p}, \frac{1}{q} \right) \in [0,1] \times [0,1]$ such that 
$\| T_{\mu}f \|_{L^{q}(\mathbb{R}^{4})} \leq C \| f \|_{L^{p}(\mathbb{R}^{4})}$ for all $f \in \mathcal{S}(\mathbb{R}^{4})$ and for some positive constant $C$ depending only on $p$ and $q$. We say that the measure $\mu$ is $L^{p}$-improving if $E_{\mu}$ does not reduce to the diagonal $\frac{1}{p} = \frac{1}{q}$. 

\

The following proposition gives a necessary condition for the $L^{p}(\mathbb{R}^{4}) - L^{q}(\mathbb{R}^{4})$ boundedness of the operator 
$T_{\mu}$. This result is a particular case of Proposizione 2.2 in \cite{Fulvio} applied to the measure $\mu$ given by (\ref{mu2}). For readers' convenience, we include a proof.

\begin{proposition} \label{prop restrict mu}
If $\left( \frac{1}{p}, \frac{1}{q} \right) \in E_{\mu}$, then 
$\frac{1}{q} \geq \frac{1}{p} - \frac{\alpha_1+\alpha_2}{\alpha_1+\alpha_2 + 2m}$.
\end{proposition}

\begin{proof}
For $c < s_1 < s_2 < d$, let $Q= \{ (t^{\alpha_1}, t^{\alpha_2} s) : \frac{1}{2} \leq t \leq 1, \, s_1 \leq s \leq s_2 \}$ and 
$M = \sup \{ |\varphi(x)| : x \in Q \}$. For $\delta > 0$ and small, we define
\[
R_{\delta} = (-2\delta^{\alpha_1}, 2\delta^{\alpha_1}) \times (-2\delta^{\alpha_2}, 2\delta^{\alpha_2}) \times 
(-(2M+1)\delta^{m}, (2M+1)\delta^{m})^{2},
\]
$f_{\delta} = \chi_{R_{\delta}}$, $E_{\delta} = \delta \bullet Q$ and
\[
A_{\delta} = \{ (x', x'') \in E_{\delta} \times \mathbb{R}^{2} : \| x'' - \varphi(x') \|_{\mathbb{R}^{2}} \leq \delta \}.
\]

To prove the proposition it suffices to show that there exists a positive constant $C$ independent of $\delta$ such that 
$| \mu \ast f_{\delta}(x', x'')| \geq C \delta^{\alpha_1 + \alpha_2}$. Indeed,
\[
\mu \ast f_{\delta}(x', x'') = \int_{V^{c,d}_{1}}  \chi_{R_{\delta}}(x' - y, x'' - \varphi(y)) \, dy.
\]
Then, for every $(x', x'') \in A_{\delta}$ fixed, from the homogeneity of $\varphi$, it follows that
$(x' - y, x'' - \varphi(y)) \in R_{\delta}$ for all $y \in E_{\delta}$. So,
\[
| \mu \ast f_{\delta}(x', x'')| \geq |E_{\delta}| = \delta^{\alpha_1 + \alpha_2} |Q|, \,\,\, \text{for all} \,\, (x', x'') \in A_{\delta},
\]
and therefore 
\[
\| \mu \ast f_{\delta} \|_{q} \geq \left( \int_{A_{\delta}} |\mu \ast f_{\delta}|^{q} \right)^{1/q} \geq C \delta^{\alpha_1 + \alpha_2} |A_{\delta}|^{1/q} = C \delta^{\alpha_1 + \alpha_2 + \frac{\alpha_1 + \alpha_2  + 2m}{q}},
\]
where $C = |Q|$. On the other hand, $(\frac{1}{p}, \frac{1}{q}) \in E_{\mu}$ implies that 
$\| \mu \ast f_{\delta} \|_{q} \leq C_{p,q} \| f_{\delta} \|_{p} \leq C_{p,q} \delta^{\frac{\alpha_1 + \alpha_2  + 2m}{p}}$.
Thus $\delta^{\alpha_1 + \alpha_2 + \frac{\alpha_1 + \alpha_2  + 2m}{q}} \leq C \delta^{\frac{\alpha_1 + \alpha_2  + 2m}{p}}$ for all
$\delta > 0$ and small. This implies, taking $\delta$ small enough, that $\frac{1}{q} \geq \frac{1}{p} - \frac{\alpha_1+\alpha_2}{\alpha_1+\alpha_2 + 2m}$.
\end{proof}

The following theorem states that the measure $\mu$ given by (\ref{mu2}) is $L^{p}$-improving.

\begin{theorem} \label{Type set thm}
For $p = \frac{\alpha_1+\alpha_2+2m}{\alpha_1+\alpha_2+m}$, the closed triangle with vertices $(0,0)$, $(1,1)$ and 
$\left(\frac{1}{p}, \frac{1}{p'} \right)$ is contained in $E_{\mu}$ and $\left(\frac{1}{p}, \frac{1}{p'} \right) \in \partial E_{\mu}$.
\end{theorem}

\begin{proof}
It is known that if $\left( \frac{1}{p}, \frac{1}{q} \right) \in E_{\mu}$, then $\frac{1}{q} \leq \frac{1}{p}$ and 
$\frac{1}{q} \geq \frac{2}{p} -1$ (cf. \cite{SteinWeiss}, pp. 33; and \cite{Oberlin}, Theorem 1); from Proposition \ref{prop restrict mu}, we also have that $\frac{1}{q} \geq \frac{1}{p} - \frac{\alpha_1+\alpha_2}{\alpha_1+\alpha_2 + 2m}$. In particular, these facts imply that if $\left( \frac{1}{p}, \frac{1}{p'} \right) \in E_{\mu}$, then $\frac{1}{2} \leq \frac{1}{p} \leq
\frac{\alpha_1+\alpha_2+m}{\alpha_1+\alpha_2+2m}$. On the other hand, by Riesz-Thorin Theorem, $E_{\mu}$ is a convex set and, since $\mu$ is
finite, the diagonal $\frac{1}{q} = \frac{1}{p}$ is contained in $E_{\mu}$. Thus, to prove the theorem it is enough to see that
$\left(\frac{1}{p}, \frac{1}{p'} \right) \in E_{\mu}$ for $p = \frac{\alpha_1+\alpha_2+2m}{\alpha_1+\alpha_2+m}$. Consider now, for 
$z \in \mathbb{C}$, the analytic family of distributions $J_z$ defined by (\ref{Jz}). For $z$ in the strip 
$-\frac{\alpha_1+\alpha_2}{m} \leq \Re(z) \leq 2$, let $S_z$ be the analytic family of operators defined by
\[
S_z f = \mu \ast J_z \ast f, \,\,\,\,\,\, f \in \mathcal{S}(\mathbb{R}^{4}).
\]
From Theorem \ref{muhat estimate}, we have $| \widehat{\mu}(\xi', \xi'')| \leq C |\xi''|^{-\frac{\alpha_1+\alpha_2}{m}}$, and since 
$\widehat{J_z} = 1 \otimes 4 \pi I_{2-z}$ we obtain that
\[
\| S_z \|_{L^{2}(\mathbb{R}^{4}) \to L^{2}(\mathbb{R}^{4})} \leq \left\| \widehat{\mu} \widehat{J_z} \right\|_{\infty} \leq C 
\left| \frac{2^{\frac{z}{2}+1}}{\Gamma\left(1-\frac{z}{2}\right)} \right| \,\,\,\,\,\, \text{for} \,\, \Re(z) = -\frac{\alpha_1+\alpha_2}{m},
\]
where $C$ does not depend on $z$. Since
$(\mu \ast J_z)(x',x'') = \frac{2^{-\frac{z}{2}}}{\pi \Gamma \left( \frac{z}{2} \right)} |x'' - \varphi(x')|^{z-2}$ 
for $x', x'' \in \mathbb{R}^{2}$, we obtain
\[
\| S_z \|_{L^{1}(\mathbb{R}^{4}) \to L^{\infty}(\mathbb{R}^{4})} = \| \mu \ast J_z \|_{\infty} \leq   
\left| \frac{2^{-\frac{z}{2}}}{\pi \Gamma \left( \frac{z}{2} \right)} \right| \,\,\,\,\,\, \text{for}   \,\, \Re(z) = 2.
\]
As in Theorem \ref{Restrict thm}, the family $\{ S_z : -\frac{\alpha_1+\alpha_2}{m} \leq \Re(z) \leq 2 \}$ satisfies the hypotheses of the complex interpolation theorem and then, since $-\frac{\alpha_1+\alpha_2}{m} t+ 2 (1-t)=0$ for $t= \frac{2m}{\alpha_1+\alpha_2+2m}$, 
$S_0$ is a bounded operator from $L^{p}(\mathbb{R}^{4})$ into $L^{p'}(\mathbb{R}^{4})$ for 
$\frac{1}{p} = \frac{t}{2} + (1-t) = \frac{\alpha_1+\alpha_2+m}{\alpha_1+\alpha_2+2m}$. Finally, being $S_0 = T_{\mu}$ the theorem follows.
\end{proof}

\begin{remark}
If one consider $\varphi \in C^{\infty}(V^{a,b})$ with additional hypotheses in the lemma \ref{eigen functions}, then the theorems 
\ref{Restrict thm} and \ref{Type set thm} are still true. We observe that the lemma \ref{eigen functions} still holds if we replace the word analytic by smooth, the proof depends on certain "well-controlled" monotonicity conditions given in terms of some derivatives of fixed order being non-vanishing. The remain of the lemmas hold considering only $\varphi \in  C^{\infty}(V^{a,b})$. We leave the details to the interested reader.
\end{remark}

\section{An example}

Let $\varphi : \mathbb{R}^{2} \to \mathbb{R}^{2}$ defined by
\[
\varphi(x,y) = (\varphi_1(x,y), \varphi_2(x,y)) = \left( \frac{-4}{132}x^{12}+\frac{1}{30}y^{6}, \, \frac{-3}{132}x^{12}+ \frac{1}{30}y^{6}
+ \frac{5}{90}x^{6}y^{3} \right), \,\,\,\,\,\, (x,y) \in \mathbb{R}^{2}.
\]
In this case we have that $\varphi(t \bullet (x,y)) = t^{m}\varphi(x,y)$ with $\alpha_1= \frac{1}{2}$, $\alpha_2=1$ and $m = 6$, so H1 
and H2 hold. A computation shows that the discriminant of the quadratic form
\[
(\zeta_1, \zeta_2) \to \det\left( \zeta_1 \varphi_{1}''(1,y) + \zeta_2 \varphi_{2}''(1,y) \right) = \langle K(1,y) \, \zeta, \,\zeta \rangle
\]
is
\[
\det K(1,y)= -4y^{4} \left(-3+\frac{5}{3}y^{3}\right) \left(y^{4}+\frac{y}{3} \right) + 4y^{8} - \frac{1}{4} 
\left[ 4 \left(y^{4}+\frac{y}{3} \right) +  y^{4}\left(3-\frac{5}{3}y^{3}\right) \right]^{2}.
\]
We observe that $k_{11}(1,y) = -4y^{4}$, and the discriminant is positive for every $y \in \left[\frac{97}{100}, 1 \right)$ and it has a simple zero for $y=1$. By Lemma \ref{eigen functions}, we have that the eigenvalues $\Lambda_1(1, \cdot \, )$ and $\Lambda_2(1, \cdot \, )$ of
$K(1, \cdot \, )$ are real analytic on $(\frac{97}{100}, 1)$. Since $\Lambda_1(1,1) \Lambda_2(1,1) = \det K(1,1) = 0$ and 
$\left[ \frac{d}{dy} \det K(1,y) \right]_{y=1}\neq 0$, it follows that $\Lambda_1(1,1) \neq 0$ and 
$\left[ \frac{d}{dy} \Lambda_2(1,y) \right]_{y=1}\neq 0$ or $\Lambda_2(1,1) \neq 0$ and 
$\left[ \frac{d}{dy} \Lambda_1(1,y) \right]_{y=1}\neq 0$. Thus, H3 holds with $\sigma =1$. It is clear that H4 also holds. Finally H1 - H4 hold for  $\varphi$ and $V = \left\{ (t^{1/2}, t y) : \frac{97}{99} < y < 1 \, \text{and} \, 0 < t < 1 \right\}$.

\

{\bf Acknowledgements.} We are very grateful to the referee for the careful reading of the paper and for the numerous useful comments and suggestions which helped us to improve the original manuscript.

\end{document}